\documentclass[12pt]{amsart}
\usepackage{amsmath,amsthm,latexsym,amscd,amsbsy,amssymb,amsfonts,fleqn,leqno,
euscript, pb-diagram}

\let\SavedRightarrow=\Rightarrow
\usepackage{marvosym}
\let\Rightarrow=\SavedRightarrow

\newcommand{\Ia }{\mathbb I}
\newcommand{\Aa }{\mathbb A}
\newcommand{\Ma }{\mathbb M}
\newcommand{\Ga }{\mathbb G}

\renewcommand{\int}{\operatorname{Int}}

\newtheorem{thm}{Theorem}[section]
\newtheorem{pro}[thm]{Proposition}
\newtheorem{lem}[thm]{Lemma}

\title{The  monoid  consisting of Kuratowski operations}

\author{Szymon Plewik}
\address{Institute of Mathematics, University of Silesia, ul. Bankowa 14, 40-007 Katowice}
\email{plewik@math.us.edu.pl}

\author{Marta Walczy\'nska}
\address{Institute of Mathematics, University of Silesia, ul. Bankowa 14, 40-007 Katowice}
\email{mwalczynska@us.edu.pl}

\begin{document}

\maketitle

\begin{abstract} The paper  fills  gaps in knowledge about Kuratowski operations  which are already in the literature.
The Cayley table for these operations has been drawn up. Techniques, using only paper and pencil,   to point out all  semigroups and its isomorphic types are applied.  Some  results apply only to  topology, one can not bring them out, using only  properties of  the complement  and a closure-like operation. The arguments are by systematic study of possibilities.
\end{abstract}

\section{Introduction} 
Let $X$ be a topological space. Denote by $A^-$  closure of the set $A \subseteq X$. Let $A^c$ be the complement of $A$, i.e. $X \setminus A =A^c$. The aim of this note is to examine monoids generated under compositions from the closure and the complement. A widely known fact due to K. Kuratowski \cite{1922} states that at most 14 distinct operations can be formed such compositions.  
  Mark them as  follows.
\[
\begin{tabular}{ll} \textit{Kuratowski operations}:\\
$\sigma_0(A) = A $ (\textit{the identity}),  &   $\sigma_1(A)=A^{c}$ (\textit{the complement}), \\ 

$\sigma_2(A) =   A^{-} $ (\textit{the closure}), & $\sigma_3(A)=A^{c-}$, \\ 

$\sigma_4(A) = A^{-c}$, & $ \sigma_5(A)=A^{c-c}$ (\textit{the interior}), \\ 

$\sigma_6(A) = A^{-c-}$,& $ \sigma_7(A)=A^{c-c-}$,\\ 

$\sigma_8(A) =A^{-c-c}$, & $ \sigma_9(A)=A^{c-c-c}$, \\ 

$\sigma_{10}(A) =A^{-c-c-}$, &$ \sigma_{11}(A)=A^{c-c-c-}$, \\ 

$\sigma_{12}(A) =A^{-c-c-c}$,&$\sigma_{13}(A)=A^{c-c-c-c}$. \\ 
&\\
\textit{Cancellation rules}:\\ $ A^{-c-}  =A^{-c-c-c-}$, & $A^{c-c-}=A^{c-c-c-c-}$.  
\end{tabular}
\] 

Kuratowski operations have been studied by several authors, for example \cite{chap} or \cite{sh}. A list of some other  authors one can find in the  paper \cite{gj} by B. J. Gardner and M. Jackson. For the first time these operations  were systematically studied in the dissertation by K. Kuratowski, whose results were published in \cite{1922}. Tasks relating to these operations are usually resolved at  lectures or  exercises  with General Topology. They are normally left to students for independent resolution. For example, determine how many different ways they convert a given  set.

This note is organized as follows. Kuratowski operations and their marking are described in the introduction. Their properties of a much broader context than  for  topologies are presented in Part 2. The Cayley table, for the monoid $\Ma$ of all Kuratowski operations,  has been drawn up in Part 3. We hope that  this table has not yet been published in the literature.  Having this table, one can create a computer program that calculates all the semigroups contained in $\Ma$.  However, in parts 4 - 8, we present a framework (i.e.,  techniques using only paper and pencil)  to point out all 118 semigroups and 56 isomorphic types of them. 
The list of 43 semigroups which are not monoids is presented in Part 9. In this part,  also are discussed isomorphic types in order of the  number of elements in  semigroups. Finally, we  present cancellation rules (relations)  motivated by some  topological spaces.

\section{Cancellation rules} 
A map $f: P(X) \to  P(X)$ is called:
\begin{itemize}
	\item  \textit{increasing}, if $A\subseteq B$ implies  $f(A)\subseteq f(B)$;
\item \textit{decreasing}, if  $A\subseteq B$ implies  $f(B)\subseteq f(A)$;
\item \textit{an involution}, if the composition  $f \circ f$ is the identity; 
\item \textit{an idempotent}, if  $f \circ f =f$.\end{itemize}
Assume that $A \mapsto\sigma_0(A) $ is the identity,  $A \mapsto \sigma_1(A)$ is a decreasing involution and $A\mapsto \sigma_2(A)$ is an increasing idempotent map. Other operations $\sigma_i$ let be compositions of $\sigma_1$ and $\sigma_2$ as in the above table.  We get the following cancellation rules:

\begin{lem}\label{11}  If   $B \subseteq \sigma_2(B)$, then $$\sigma_2 \circ \sigma_{12} = \sigma_6 \mbox{ and } \sigma_2 \circ \sigma_{13} = \sigma_7 .$$  
\end{lem}
\begin{proof} For clarity of this proof,  use  designations $\sigma_1(A)=A^c$ and $\sigma_2(A)=A^-$. Thus, we shall to prove  $$ A^{-c-c-c-}=A^{-c-} \mbox{  and }  A^{c-c-c-c-}= A^{c-c-}.$$ We start with $A^{-c-c} \subseteq A^{-c-c-}$, substituting $B=A^{-c-c}$ in $B \subseteq B^-$.  This corresponds to $A^{-c-c-c} \subseteq A^{-c-cc}=A^{-c-}$, since $\sigma_1$ is a decreasing involution. Hence $A^{-c-c-c-} \subseteq A^{-c-}$, since $\sigma_2$  an increasing  idempotent.

Since $\sigma_1$ is decreasing, $\sigma_2$ is increasing and   $B \subseteq B^-$ we have  $$B^{-c} \subseteq B^{c}\subseteq B^{c-}.$$   Thus   $A^{-c-c} \subseteq A^{-cc-}=A^-$, if we  put  $A^{-c} = B$. Again using that   $\sigma_2$ is increasing and $\sigma_1$ is decreasing,  we obtain  $ A^{-c-c-} \subseteq A^{-}$ and  then $A^{-c}\subseteq A^{-c-c-c}  $. Finally,  we  get   $A^{-c-}\subseteq A^{-c-c-c-} $.  

  With $A^c$ in the place of  $A$ in the rule  $A^{-c-} = A^{-c-c-c-}$ we get the second rule.
\end{proof}

 In academic textbooks of general topology, for example \cite[Problem 1.7.1.]{eng}, one can  find a hint suggested to prove above cancellation rules.   Students go like this: Steps $A \subseteq A^-$ and  $A^c \subseteq A^{c-} $ lead to $A^{c-c} \subseteq A$;    The special case $  A^{-c-c-c} \subseteq A^{-c-}$ (of  $A^{c-c} \subseteq A$) leads to $  A^{-c-c-c-} \subseteq A^{-c-}$;  Steps $A^{-c-c}  \subseteq A^-$ and  $A^{-c-c-} \subseteq A^{-} $ lead to $A^{-c-} \subseteq A^{-c-c-c-}$; Use the last step of the proof of Lemma \ref{11} at the end. Note that, above proofs does not use  axioms of topology: 
\begin{itemize}
	\item 
 $\emptyset = \emptyset^-$; \item $(A \cup B)^-= A^- \cup B^-$.
 \end{itemize}
 In the literature there are articles in which Kuratowski operations are replaced by some other  mappings. For example, W. Koenen \cite{ko} considered linear spaces and put $\sigma_2(A)$ to be the convex hull of $A$. In fact, S. Shum \cite{shu} considered   $\sigma_2$ as  the closure  due to the algebraic operations.  Add to this, that these operations can be applied to so called Fr\'echet (V)spaces, which were considered in the book \cite[p. 3 - 37]{sie}.

\section{The monoid $\mathbb M$} 

Let $\mathbb M$ be the monoid consisting  of all Kuratowski operations, i.e. there are assumed cancellation rules: $\sigma_6 = \sigma_1 \circ \sigma_{12}$ and $\sigma_7 = \sigma_1 \circ \sigma_{13}$. 
Fill in the Cayley table for $\mathbb M$, where the row and column marked by the identity  are omitted.  Similarly as in \cite{ca},  the factor that labels the row comes first, and that the factor that labels the column is second. For example, $\sigma_i\circ \sigma_k$ is in the  row marked by $\sigma_i$ and the column marked by $\sigma_k$.  

$$ 
\begin{tabular}{l|lllllllllllll}
 
 &$\sigma_{1}$& $\sigma_{2}$  &$\sigma_{3}$  &$\sigma_{4}$ &$\sigma_{5}$ &$\sigma_{6}$ &$\sigma_{7}$ &$\sigma_{8}$ &$\sigma_{9}$ &$\sigma_{10}$ & $\sigma_{11}$ &$\sigma_{12}$ &$\sigma_{13}$   \\   \hline

 $\sigma_{1}$ & $\sigma_{0}$ & $\sigma_{4}$ &$\sigma_{5}$ &$\sigma_{2}$ &$\sigma_{3}$ &$\sigma_{8}$ &$\sigma_{9}$ &$\sigma_{6}$ &$\sigma_{7}$ &$\sigma_{12}$ & $\sigma_{13}$ &$\sigma_{10}$ &$\sigma_{11}$   \\ 
 
 $\sigma_{2}$ & $\sigma_{3}$ & $\sigma_{2}$ &$\sigma_{3}$ &$\sigma_{6}$ &$\sigma_{7}$ &$\sigma_{6}$ &$\sigma_{7}$ &$\sigma_{10}$ &$\sigma_{11}$ &$\sigma_{10}$ & $\sigma_{11}$ &$\sigma_{6}$ &$\sigma_{7}$   \\ 

 $\sigma_{3}$ & $\sigma_{2}$ & $\sigma_{6}$ &$\sigma_{7}$ &$\sigma_{2}$ &$\sigma_{3}$ &$\sigma_{10}$ &$\sigma_{11}$ &$\sigma_{6}$ &$\sigma_{7}$ &$\sigma_{6}$ & $\sigma_{7}$ &$\sigma_{10}$ &$\sigma_{11}$   \\ 

 $\sigma_{4}$ & $\sigma_{5}$ & $\sigma_{4}$ &$\sigma_{5}$ &$\sigma_{8}$ &$\sigma_{9}$ &$\sigma_{8}$ &$\sigma_{9}$ &$\sigma_{12}$ &$\sigma_{13}$ &$\sigma_{12}$ & $\sigma_{13}$ &$\sigma_{8}$ &$\sigma_{9}$   \\ 

 $\sigma_{5}$ & $\sigma_{4}$ & $\sigma_{8}$ &$\sigma_{9}$ &$\sigma_{4}$ &$\sigma_{5}$ &$\sigma_{12}$ &$\sigma_{13}$ &$\sigma_{8}$ &$\sigma_{9}$ &$\sigma_{8}$ & $\sigma_{9}$ &$\sigma_{12}$ &$\sigma_{13}$   \\ 

 $\sigma_{6}$ & $\sigma_{7}$ & $\sigma_{6}$ &$\sigma_{7}$ &$\sigma_{10}$ &$\sigma_{11}$ &$\sigma_{10}$ &$\sigma_{11}$ &$\sigma_{6}$ &$\sigma_{7}$ &$\sigma_{6}$ & $\sigma_{7}$ &$\sigma_{10}$ &$\sigma_{11}$   \\ 

 $\sigma_{7}$ & $\sigma_{6}$ & $\sigma_{10}$ &$\sigma_{11}$ &$\sigma_{6}$ &$\sigma_{7}$ &$\sigma_{6}$ &$\sigma_{7}$ &$\sigma_{10}$ &$\sigma_{11}$ &$\sigma_{10}$ & $\sigma_{11}$ &$\sigma_{6}$ &$\sigma_{7}$   \\ 

 $\sigma_{8}$ & $\sigma_{9}$ & $\sigma_{8}$ &$\sigma_{9}$ &$\sigma_{12}$ &$\sigma_{13}$ &$\sigma_{12}$ &$\sigma_{13}$ &$\sigma_{8}$ &$\sigma_{9}$ &$\sigma_{8}$ & $\sigma_{9}$ &$\sigma_{12}$ &$\sigma_{13}$   \\ 

 $\sigma_{9}$ & $\sigma_{8}$ & $\sigma_{12}$ &$\sigma_{13}$ &$\sigma_{8}$ &$\sigma_{9}$ &$\sigma_{8}$ &$\sigma_{9}$ &$\sigma_{12}$ &$\sigma_{13}$ &$\sigma_{12}$ & $\sigma_{13}$ &$\sigma_{8}$ &$\sigma_{9}$   \\ 

 $\sigma_{10}$ & $\sigma_{11}$ & $\sigma_{10}$ &$\sigma_{11}$ &$\sigma_{6}$ &$\sigma_{7}$ &$\sigma_{6}$ &$\sigma_{7}$ &$\sigma_{10}$ &$\sigma_{11}$ &$\sigma_{10}$ & $\sigma_{11}$ &$\sigma_{6}$ &$\sigma_{7}$   \\ 

 $\sigma_{11}$ & $\sigma_{10}$ & $\sigma_{6}$ &$\sigma_{7}$ &$\sigma_{10}$ &$\sigma_{11}$ &$\sigma_{10}$ &$\sigma_{11}$ &$\sigma_{6}$ &$\sigma_{7}$ &$\sigma_{6}$ & $\sigma_{7}$ &$\sigma_{10}$ &$\sigma_{11}$   \\ 

 $\sigma_{12}$ & $\sigma_{13}$ & $\sigma_{12}$ &$\sigma_{13}$ &$\sigma_{8}$ &$\sigma_{9}$ &$\sigma_{8}$ &$\sigma_{9}$ &$\sigma_{12}$ &$\sigma_{13}$ &$\sigma_{12}$ & $\sigma_{13}$ &$\sigma_{8}$ &$\sigma_{9}$   \\ 

 $\sigma_{13}$ & $\sigma_{12}$ & $\sigma_{8}$ &$\sigma_{9}$ &$\sigma_{12}$ &$\sigma_{13}$ &$\sigma_{12}$ &$\sigma_{13}$ &$\sigma_{8}$ &$\sigma_{9}$ &$\sigma_{8}$ & $\sigma_{9}$ &$\sigma_{12}$ &$\sigma_{13}$  
\end{tabular} 
$$

It turns out that the above table  allows us describe all semigroups contained in  $\mathbb M$, using   pencil-and-paper techniques,  only. The argument will be by a systematic case study of possibilities. Preparing the list of all semigroups consisting of Kuratowski operations we used following principles: 
\begin{itemize}
	\item 
Minimal collection of generators is written using $ \left\langle A,B, \ldots, Z\right\rangle $, where letters denote  generators; \item When a semigroup has a few minimal collections  of generators, then its name is the first collection in the dictionary order;   \item All minimal collections of generators are written with the exception of some  containing $\sigma_0$;\item  We  leave to the readers verification  that our list is complete, sometimes we add  hints. 
\end{itemize}

\section{Semigroups with $\sigma_1$}

Observe that each semigroup which contains $\sigma_0$ is a monoid. Since  
$\sigma_{1}\circ \sigma_{1}= \sigma_{0}$,  a semigroup which  contains $\sigma_1$ is a monoid, too. 

\begin{thm}\label{t1}
There are three monoids containing $\sigma_1$: 
\begin{enumerate}
\item $ \left\langle \sigma_{1}\right\rangle = \{\sigma_{0}, \sigma_{1}\} $;
	\item  $\Ma = \left\langle\sigma_{1}, \sigma_{i}\right\rangle =  \{ \sigma_0, \sigma_{1},  \ldots , \sigma_{13}\}$, where $i \in \{2,3,4,5\}$;  \item Let  $\Ma_1 =   \left\langle\sigma_{1}, \sigma_{6}\right\rangle$. If  $j \in \{6,7, \ldots ,13\}$, then   $$\Ma_1 =   \left\langle\sigma_{1}, \sigma_{j}\right\rangle = \{\sigma_0, \sigma_1 \}  \cup \{  \sigma_{6}, \sigma_7 , \ldots , \sigma_{13}\}.$$ 
\end{enumerate}
\end{thm}
\begin{proof}  The equality $(1)$  is obvious. 

  Since  $\sigma_2= \sigma_3 \circ \sigma_1= \sigma_1 \circ \sigma_4 = \sigma_1 \circ \sigma_5 \circ \sigma_1,$ we have $ \left\langle\sigma_{1}, \sigma_{2}\right\rangle =\left\langle\sigma_{1}, \sigma_{3}\right\rangle = \left\langle\sigma_{1}, \sigma_{4}\right\rangle = \left\langle\sigma_{1}, \sigma_{5}\right\rangle= \Ma.$ 
  
If $j \in \{6,7, \ldots ,13\}$, then any composition $\sigma_j \circ \sigma_i$ or $\sigma_k \circ \sigma_j$  belongs to $\Ma_1$, and so     $ \left\langle\sigma_{1}, \sigma_{j}\right\rangle \subseteq \Ma_1$. We have $\sigma_7= \sigma_6 \circ \sigma_1$, $\sigma_8= \sigma_1 \circ \sigma_6$, $\sigma_9= \sigma_1 \circ \sigma_7$, $\sigma_{10}= \sigma_6 \circ \sigma_6$, $\sigma_{11}= \sigma_6 \circ \sigma_7$, $\sigma_{12}= \sigma_8 \circ \sigma_6$ and $\sigma_{13}= \sigma_8 \circ \sigma_7$, and so $ \Ma_1= \left\langle\sigma_{1}, \sigma_{6}\right\rangle = \{\sigma_0, \sigma_1 \}  \cup \{  \sigma_{6}, \sigma_7 , \ldots , \sigma_{13}\}.$

Since  $\sigma_6= \sigma_7 \circ \sigma_1 = \sigma_1 \circ \sigma_8 = \sigma_1 \circ \sigma_9 \circ \sigma_1 =
\sigma_{10} \circ \sigma_1 \circ \sigma_{10} =\sigma_{11} \circ \sigma_{11} \circ \sigma_1   =\sigma_1 \circ \sigma_{12} \circ \sigma_{12}$ and $ \sigma_6 =  \sigma_1 \circ \sigma_{13} \circ \sigma_1 \circ \sigma_{13} \circ \sigma_1$, we have $ \left\langle\sigma_{1}, \sigma_{j}\right\rangle = \Ma_1$  for  $j \in \{6,7, \ldots ,13\}$.
\end{proof}

Consider the permutation  
$$  \left(
\begin{tabular}{llllllllllllll}
                  $\sigma_{0}$ &$\sigma_{1}$& $\sigma_{2}$ &$\sigma_{3}$&$\sigma_{4}$&$\sigma_{5}$&$\sigma_{6}$&$\sigma_{7}$&$\sigma_{8}$&$\sigma_{9}$  &$\sigma_{10}$& $\sigma_{11}$ &$\sigma_{12}$  &$\sigma_{13}$  \\ 
$\sigma_{0}$ &$\sigma_{1}$& $\sigma_{5}$ &$\sigma_{4}$&$\sigma_{3}$&$\sigma_{2}$ &$\sigma_{9}$&$\sigma_{8}$ &$\sigma_{7}$&$\sigma_{6}$  &$\sigma_{13}$& $\sigma_{12}$ &$\sigma_{11}$  &$\sigma_{10}$ 
\end{tabular}\right).
$$ It determines  the automorphism   $\Aa: \Ma \to \Ma$. 
\begin{thm} The identity and $\Aa$ are the only automorphisms of $\Ma$.
\end{thm}
\begin{proof} Delete  rows and columns marked by  $\sigma_1$ in the Cayley table for $\mathbb M$. Then, check that the operation $\sigma_3$ is in  the row or the column marked by  $\sigma_3$ only. Also,   the operation  $\sigma_4$ is in  the row or the column marked by  $\sigma_4$ only. Therefore the semigroup    $$<\sigma_{3}, \sigma_{4}> = \{\sigma_{2}, \sigma_{3}, \ldots , \sigma_{13}\},$$ has a unique minimal set of generators $\{\sigma_{3}, \sigma_{4} \}.$  The reader is left to check this with the Cayley table for $\mathbb M$.  

Suppose  $\Ga$ is an  automorphism of $\Ma$. By Theorem \ref{t1}, $\Ga$  transforms the set $\{ \sigma_2, \sigma_3, \sigma_4,\sigma_5 \}$ onto itself. However $\sigma_2$ and  $\sigma_5$ are idempotents, but $\sigma_3$ and $\sigma_4$ are not idempotents. So, there are two possibilities:  $\Ga (\sigma_3) =\sigma_3$ and $\Ga (\sigma_4) =\sigma_4$,  which implies that $\Ga$ is the identity; $\Ga (\sigma_3) =\sigma_4$ and $\Ga (\sigma_4) =\sigma_3$,  which implies $\Ga=\Aa$. 
The reader is left to check this with the Cayley table for $\mathbb M$. We offer hints: $\sigma_{2}= \sigma_{3} \circ \sigma_{4}$,     $\sigma_{5}= \sigma_{4} \circ \sigma_{3}$,  $\sigma_{6}= \sigma_{3} \circ \sigma_{2}$,  $\sigma_{7}= \sigma_{3} \circ \sigma_{3}$,  $\sigma_{8}= \sigma_{4} \circ \sigma_{4}$,  $\sigma_{9}= \sigma_{4} \circ \sigma_{5}$,  $\sigma_{10}= \sigma_{6} \circ \sigma_{6}$, $\sigma_{11}= \sigma_{6} \circ \sigma_{7}$, $\sigma_{12}= \sigma_{8} \circ \sigma_{6}$    and  $\sigma_{13}= \sigma_{8} \circ \sigma_{7}$; to verify the details of this proof. 
\end{proof}

\section{The monoid of all idempotents}
The set $\{ \sigma_{0},   \sigma_{2}, \sigma_{5},  \sigma_{7},  \sigma_{8}, \sigma_{10}, \sigma_{13}\}$ consists  of all squares in $\Ma$. These squares are idempotents and lie on the main diagonal in the Cayley table for $\Ma$. They constitute the monoid and $$ \{ \sigma_{0},   \sigma_{2}, \sigma_{5},  \sigma_{7},  \sigma_{8}, \sigma_{10}, \sigma_{13}\} = \left\langle \sigma_{0},   \sigma_{2}, \sigma_{5} \right\rangle.$$ The permutation  
$$ \left(
\begin{tabular}{llllllll}
                   $\sigma_{0}$ & $\sigma_{2}$ &$\sigma_{5}$&$\sigma_{7}$  &$\sigma_{8}$  &$\sigma_{10}$       &$\sigma_{13}$  \\ 
                   $\sigma_{0}$ & $\sigma_{2}$  & $\sigma_{5}$ & $\sigma_{8}$ &$\sigma_{7}$  & $\sigma_{10}$   &$\sigma_{13}$  
\end{tabular}\right)
$$ determines the bijection   $\Ia: \left\langle \sigma_{0},   \sigma_{2}, \sigma_{5} \right\rangle \to \left\langle \sigma_{0},   \sigma_{2}, \sigma_{5} \right\rangle $ such that   $$\Ia(\alpha \circ \beta) = \Ia( \beta) \circ \Ia(\alpha ),$$ for any $\alpha, \beta \in \left\langle \sigma_{0},   \sigma_{2}, \sigma_{5} \right\rangle$. To verify this,   
apply   equalities  $\sigma_2\circ \sigma_5= \sigma_7 $,  $\sigma_5\circ \sigma_2= \sigma_8 $, $\sigma_2\circ \sigma_5\circ \sigma_2 = \sigma_{10} $ and  $\sigma_5\circ \sigma_2\circ \sigma_5 = \sigma_{13} $. Any bijection   $\Ia: G \to H$, having property $\Ia(\alpha \circ \beta) = \Ia( \beta) \circ \Ia(\alpha )$,  transposes Cayley tables for  semigroups $G$ and $H$.  The readers can check that the  few semigroups discussed below have this property.

We shall classify all semigroups contained in the semigroup $\left\langle  \sigma_{2}, \sigma_{5} \right\rangle$. 
Every such semigroup can be extended to a monoid  by attaching $\sigma_0$ to it. This gives a complete classification of all semigroups in  $\left\langle \sigma_{0},   \sigma_{2}, \sigma_{5} \right\rangle $. 

The 
 semigroup $\left\langle  \sigma_{2}, \sigma_{5} \right\rangle$ contains six groups with exactly one element. 

Semigroups $\left\langle \sigma_{2},   \sigma_{10} \right\rangle =\{\sigma_{2},   \sigma_{10}\} $ and $\left\langle \sigma_{5},   \sigma_{13} \right\rangle =  \{\sigma_{5},   \sigma_{13}\}$  are monoids. Both consist of exactly  two  elements and are not groups, so they are isomorphic.

Semigroups  $\left\langle \sigma_{7},   \sigma_{10} \right\rangle$ and $\left\langle \sigma_{8},   \sigma_{13} \right\rangle$ are isomorphic, in particular  $\Aa [\left\langle \sigma_{7},   \sigma_{10} \right\rangle] = \left\langle \sigma_{8},   \sigma_{13} \right\rangle$. Also, semigroups     $\left\langle \sigma_{7},   \sigma_{13} \right\rangle$ and  $\left\langle \sigma_{8},   \sigma_{10} \right\rangle$ are isomorphic  by  $\Aa$.  Every of these four semigroups has exactly two elements. None of them is a monoid.  They form two types of non-isomorphic semigroups, because of bijections $\{(\sigma_7,\sigma_{10}), (\sigma_{10},\sigma_{8})\}$ and $\{(\sigma_7,\sigma_{8}), (\sigma_{10},\sigma_{10})\}$ are not isomorphisms. 

 Semigroups $\left\langle \sigma_{2},   \sigma_{7} \right\rangle =\{\sigma_{2},   \sigma_{7},   \sigma_{10}\}$ and $\left\langle \sigma_{5},   \sigma_{8} \right\rangle =\{\sigma_{5},   \sigma_{8},   \sigma_{13}\}$ are isomorphic. Also, semigroups    $\left\langle \sigma_{2},   \sigma_{8} \right\rangle=\{\sigma_{2},   \sigma_{8},   \sigma_{10}\}$ and $\left\langle \sigma_{5},   \sigma_{7} \right\rangle =\{\sigma_{5},   \sigma_{7},   \sigma_{13}\}$ are isomorphic. In fact, $\Aa [\left\langle \sigma_{2},   \sigma_{7} \right\rangle] = \left\langle \sigma_{5},   \sigma_{8} \right\rangle$ and  $\Aa [\left\langle \sigma_{2},   \sigma_{8} \right\rangle] = \left\langle \sigma_{5},   \sigma_{7} \right\rangle$.   None of these semigroups is a monoid. They form two types of non-isomorphic semigroups. Indeed, any isomorphism between  $\left\langle \sigma_{2},   \sigma_{7} \right\rangle$ and $\left\langle \sigma_{2},   \sigma_{8} \right\rangle$ must be the identity on the monoid $\left\langle \sigma_{2},   \sigma_{10} \right\rangle$, and therefore would have to be the restriction of $\Ia$. But $\Ia [\left\langle \sigma_{2},   \sigma_{7} \right\rangle] = \left\langle \sigma_{2},   \sigma_{8} \right\rangle$ and $\Ia$ restricted to   $\left\langle \sigma_{7},   \sigma_{10} \right\rangle$  is not an isomorphism. 

Semigroups $\left\langle \sigma_{2},   \sigma_{13} \right\rangle =\left\langle \sigma_{2},   \sigma_{7}, \sigma_8 \right\rangle =\{\sigma_{2},   \sigma_{7},   \sigma_{8},  \sigma_{10}, \sigma_{13}\}$ and $\left\langle \sigma_{5},   \sigma_{10} \right\rangle =\left\langle \sigma_{5},   \sigma_{7}, \sigma_8 \right\rangle =\{\sigma_{5},   \sigma_{7},   \sigma_{8},  \sigma_{10}, \sigma_{13}\}$ are not monoids. They are  isomorphic by $\Aa$.

The semigroup  $\left\langle \sigma_{2},   \sigma_{5} \right\rangle$  contains exactly one semigroup with four elements  $\left\langle \sigma_{7},   \sigma_{8} \right\rangle =
\left\langle \sigma_{10},   \sigma_{13} \right\rangle =\{   \sigma_{7},   \sigma_{8},  \sigma_{10}, \sigma_{13}\}$ which is not a monoid.

Note that $\left\langle \sigma_{2},   \sigma_{5} \right\rangle$  contains  twenty different semigroups with nine non-isomorphic types. These are  
six isomorphic groups with exactly one element $\left\langle \sigma_{2}\right\rangle\cong  \left\langle \sigma_{5} \right\rangle \cong  \left\langle \sigma_{7}\right\rangle\cong  \left\langle \sigma_{8} \right\rangle \cong  \left\langle \sigma_{10} \right\rangle \cong  \left\langle \sigma_{13}\right\rangle$,
	two    isomorphic monoids with exactly two elements $\left\langle \sigma_{2},   \sigma_{10} \right\rangle\cong \left\langle \sigma_{5},   \sigma_{13} \right\rangle$, 
	two  pairs of isomorphic semigroups with exactly two elements $\left\langle \sigma_{7},   \sigma_{10} \right\rangle\cong \left\langle \sigma_{8},   \sigma_{13} \right\rangle$ and $\left\langle \sigma_{7},   \sigma_{13} \right\rangle\cong \left\langle \sigma_{8},   \sigma_{10} \right\rangle$,
two  pairs of isomorphic semigroups with exactly three elements $\left\langle \sigma_{2},   \sigma_{7} \right\rangle\cong \left\langle \sigma_{5},   \sigma_{8} \right\rangle$ and $\left\langle \sigma_{2},   \sigma_{8} \right\rangle\cong \left\langle \sigma_{5},   \sigma_{7} \right\rangle$,
	  a semigroup with exactly four elements $\left\langle \sigma_{7},   \sigma_{8} \right\rangle$ and
	two   isomorphic semigroups with exactly five elements $\left\langle \sigma_{2},   \sigma_{13} \right\rangle \cong \left\langle \sigma_{5},   \sigma_{10} \right\rangle$, and also  $\left\langle \sigma_{2},   \sigma_{5} \right\rangle$. 	
	Thus,  $\left\langle \sigma_{2},   \sigma_{5} \right\rangle$ contains twenty different semigroups with nine types of isomorphism. But $\left\langle \sigma_0,  \sigma_{2},   \sigma_{5} \right\rangle$ contains forty one different semigroups with seventeen types of isomorphism. Indeed, adding $\sigma_0$  to  semigroups  contained in $\left\langle \sigma_{2},   \sigma_{5} \right\rangle$, which are not  monoids, we get twenty monoids with  eight non-isomorphic types. 
 Adding   $\sigma_0$ to a group  contained in $  \left\langle \sigma_2, \sigma_5\right\rangle $ we get a monoid  isomorphic to  $  \left\langle \sigma_2, \sigma_{10}\right\rangle $.

 \section{The semigroup consisting of $\{ \sigma_{6}, \sigma_{7}, \ldots , \sigma_{13}\}$  }
Using the Cayley table for $\Ma$, check that  $$\Aa[ \{\sigma_{6}, \sigma_{7}, \ldots , \sigma_{13}\}] =  \{\sigma_{6}, \sigma_{7}, \ldots , \sigma_{13}\}$$ Similarly,   check that the semigroup $\left\langle\sigma_{6}, \sigma_{9}\right\rangle = \{\sigma_{6}, \sigma_{7}, \ldots , \sigma_{13}\}$ can be represented as
  $\left\langle\sigma_{6}, \sigma_{13}\right\rangle$, $\left\langle\sigma_{7}, \sigma_{12}\right\rangle$,  $\left\langle\sigma_{8}, \sigma_{11}\right\rangle$, $\left\langle\sigma_{9}, \sigma_{10}\right\rangle$ or  $\left\langle\sigma_{11}, \sigma_{12}\right\rangle$.  Also  $\left\langle\sigma_{6}, \sigma_{9}\right\rangle =\left\langle\sigma_{6}, \sigma_{7}, \sigma_{8}\right\rangle= \left\langle \sigma_{7}, \sigma_{8}, \sigma_{9}\right\rangle =\left\langle\sigma_{10}, \sigma_{11}, \sigma_{13}\right\rangle =\left\langle\sigma_{10}, \sigma_{12}, \sigma_{13}\right\rangle$. These representations  exhaust all minimal collections of Kuratowski operations which  generate $\left\langle\sigma_{6}, \sigma_{9}\right\rangle $.  Other semigroups included in $\left\langle\sigma_{6}, \sigma_{9}\right\rangle $
 have one or two minimal collection of generators. One generator have  groups
$ \left\langle\sigma_{6}\right\rangle =\{\sigma_{6}, \sigma_{10}\} $,  $ \left\langle\sigma_{9}\right\rangle =\{\sigma_{9}, \sigma_{13}\}$, 
$ \left\langle\sigma_{11}\right\rangle =\{\sigma_{7}, \sigma_{11}\} $ and  
$\left\langle\sigma_{12}\right\rangle =\{\sigma_{8}, \sigma_{12}\} $. Each of them has exactly two elements, so they are isomorphic. Semigroups $\left\langle \sigma_{7},   \sigma_{10} \right\rangle$,   $\left\langle \sigma_{7},   \sigma_{13} \right\rangle$, $\left\langle \sigma_{8},   \sigma_{10} \right\rangle$, $\left\langle \sigma_{8},   \sigma_{13} \right\rangle$  and  $\left\langle \sigma_{7},   \sigma_{8} \right\rangle  $ are discussed in the previous part. Contained in $ \left\langle\sigma_{6}, \sigma_{9}\right\rangle $ and not previously discussed semigroups  are $ \left\langle\sigma_{6}, \sigma_{7}\right\rangle $, $ \left\langle\sigma_{6}, \sigma_{8}\right\rangle $,$ \left\langle\sigma_{7}, \sigma_{9}\right\rangle $ and $ \left\langle\sigma_{8}, \sigma_{9}\right\rangle $. We leave the reader to verify that the following are all possible pairs of Kuratowski operations which constitute a minimal collection of generators for semigroups  contained in $ \left\langle\sigma_{6}, \sigma_{9}\right\rangle $, but different from the whole.
\begin{itemize}
	\item  $\left\langle\sigma_{6}, \sigma_{7}\right\rangle =
 \left\langle\sigma_{6}, \sigma_{11}\right\rangle = \left\langle\sigma_{10}, \sigma_{11}\right\rangle = \{\sigma_{6},\sigma_{7},\sigma_{10},\sigma_{11} \} $; 
\item   $\left\langle\sigma_{6}, \sigma_{8}\right\rangle = 
 \left\langle\sigma_{6}, \sigma_{12}\right\rangle = \left\langle\sigma_{10}, \sigma_{12}\right\rangle =   \{\sigma_{6},\sigma_{8},\sigma_{10},\sigma_{12} \}$; 
\item   $\left\langle\sigma_{7}, \sigma_{9}\right\rangle = 
 \left\langle\sigma_{9}, \sigma_{11}\right\rangle = \left\langle\sigma_{11}, \sigma_{13}\right\rangle =   \{\sigma_{7},\sigma_{9},\sigma_{11},\sigma_{13} \}$; 
 \item   $\left\langle\sigma_{8}, \sigma_{9}\right\rangle = 
 \left\langle\sigma_{9}, \sigma_{12}\right\rangle = \left\langle\sigma_{12}, \sigma_{13}\right\rangle =   \{\sigma_{8},\sigma_{9},\sigma_{12},\sigma_{13} \}$; 
\end{itemize}

  \begin{pro}\label{53} Semigroups  
 $\left\langle\sigma_{6}, \sigma_{7}\right\rangle$ and  $\left\langle\sigma_{8}, \sigma_{9}\right\rangle$ are isomorphic, and also semigroups  $\left\langle\sigma_{6}, \sigma_{8}\right\rangle$
and  $\left\langle\sigma_{7}, \sigma_{9}\right\rangle$ are isomorphic, but semigroups  $\left\langle\sigma_{6}, \sigma_{7}\right\rangle$,
 $\left\langle\sigma_{6}, \sigma_{8}\right\rangle$ are not isomorphic.   
\end{pro}
\begin{proof} Isomorphisms are defined by $\Aa$. Suppose $J: \left\langle\sigma_{6}, \sigma_{7}\right\rangle \to \left\langle\sigma_{6}, \sigma_{8}\right\rangle$ is an isomorphism. Thus $J[\left\langle\sigma_{7}, \sigma_{10}\right\rangle ]= \left\langle\sigma_{8}, \sigma_{10}\right\rangle$.  Given $J(\sigma_{7}) =\sigma_{8}$, we get $\sigma_{10}= J(\sigma_{10}) =J(\sigma_{7}\circ \sigma_{10})\not= \sigma_{8}\circ \sigma_{10}= \sigma_{8}$. But $J(\sigma_{7}) =\sigma_{10}$ implies $\sigma_{10}= J(\sigma_{7}) =J(\sigma_{10}\circ \sigma_{7})\not= \sigma_{8}\circ \sigma_{10}= \sigma_{8}$. Both possibilities lead to a contradiction. \end{proof}

So,  $\left\langle \sigma_{6},   \sigma_{9} \right\rangle$  contains eighteen different semigroups with eight non-isomorphic types. Indeed, these are  
four isomorphic groups with exactly one element $    \left\langle \sigma_{7}\right\rangle\cong  \left\langle \sigma_{8} \right\rangle \cong  \left\langle \sigma_{10} \right\rangle \cong  \left\langle \sigma_{13}\right\rangle$,
four  isomorphic groups with exactly two elements     $\left\langle  \sigma_{6}\right\rangle\cong  \left\langle \sigma_{9} \right\rangle \cong  \left\langle \sigma_{11} \right\rangle \cong  \left\langle \sigma_{12}\right\rangle$, 
	two pairs of isomorphic semigroups with exactly two elements $\left\langle \sigma_{7},   \sigma_{10} \right\rangle\cong \left\langle \sigma_{8},   \sigma_{13} \right\rangle$ and $\left\langle \sigma_{7},   \sigma_{13} \right\rangle\cong \left\langle \sigma_{8},   \sigma_{10} \right\rangle$ and five  semigroups with exactly four elements $\left\langle \sigma_{6},   \sigma_{7} \right\rangle \cong \left\langle \sigma_{8},   \sigma_{9} \right\rangle$, $\left\langle \sigma_{6},   \sigma_{8} \right\rangle \cong \left\langle \sigma_{7},   \sigma_{9} \right\rangle$ and $\left\langle \sigma_{7},   \sigma_{8} \right\rangle  $, and also $\left\langle \sigma_{6},   \sigma_{9} \right\rangle$.

\section{Remaining semigroups in  $\left\langle\sigma_{3}, \sigma_{4}\right\rangle$}  

 We have yet to discuss semigroups included in $\left\langle\sigma_{3}, \sigma_{4}\right\rangle$, not included in $\left\langle\sigma_{2}, \sigma_{5}\right\rangle$ and containing at least  one of Kuratowski operation $\sigma_{2}$,  $\sigma_{3}$, $\sigma_{4}$ or $\sigma_{5}$.  It will be discussed up to isomorphism $\Aa$. Obviously,   $\left\langle\sigma_{2} \right\rangle = \{ \sigma_{2}\}$ and $\left\langle\sigma_{5} \right\rangle = \{ \sigma_{5}\}$ are groups. 
 \subsection{Extensions of  $\left\langle\sigma_2\right\rangle$ and  $\left\langle\sigma_5\right\rangle$ with elements of $\left\langle\sigma_{6},  \sigma_{9}\right\rangle$}

Monoids $\left\langle\sigma_{2},  \sigma_{6}\right\rangle = \{ \sigma_{2}, \sigma_{6}, \sigma_{10}\}$ and $\left\langle\sigma_{2},  \sigma_{10}\right\rangle = \{ \sigma_{2}, \sigma_{10}\}$ have  different numbers of elements.
 Also,  $\left\langle\sigma_{2},  \sigma_{6}\right\rangle$ is isomorphic to  $\left\langle\sigma_{0},  \sigma_{6}\right\rangle$. Non-isomorphic semigroups $\left\langle\sigma_{2},  \sigma_{7}\right\rangle$ and $\left\langle\sigma_{2},  \sigma_{8}\right\rangle$ are discussed above.  
 Three following semigroups: 
\begin{itemize}
	\item 	
$\left\langle\sigma_{2},  \sigma_{11}\right\rangle = \left\langle\sigma_{2},  \sigma_{6}, \sigma_7 \right\rangle=\{\sigma_{2},  \sigma_{6}, \sigma_{7}, \sigma_{10},  \sigma_{11}\} $,  \item 
$\left\langle\sigma_{2},  \sigma_{12}\right\rangle = \left\langle\sigma_{2},  \sigma_{6}, \sigma_8 \right\rangle=\{\sigma_{2},  \sigma_{6}, \sigma_{8}, \sigma_{10},  \sigma_{12}\} $, \item   
$\left\langle\sigma_{2},  \sigma_{13}\right\rangle = \left\langle\sigma_{2},  \sigma_{7}, \sigma_8 \right\rangle=\{\sigma_{2},  \sigma_{7}, \sigma_{8}, \sigma_{10},  \sigma_{13}\}; $\end{itemize} are not monoids. They are not isomorphic. Indeed, any isomorphism between these semigroups  would lead an isomorphism between  $\left\langle\sigma_{6},  \sigma_{7}\right\rangle$, $ \left\langle\sigma_{6},  \sigma_{8}\right\rangle$ or $\left\langle\sigma_{7},  \sigma_{8}\right\rangle$. This is impossible,  by Proposition \ref{53}  and because   $\left\langle\sigma_{7},  \sigma_{8}\right\rangle$ consists of idempotents, but $\sigma_{6}$ is not an idempotent.   The nine-elements semigroup on the set $\{\sigma_2, \sigma_{6}, \sigma_{7}, \ldots , \sigma_{13}\}$ is represented as  $\left\langle\sigma_{2},  \sigma_{9}\right\rangle$. So, we have added four new semigroups, which are not isomorphic with the semigroups previously discussed. These are $\left\langle\sigma_{2},  \sigma_{11}\right\rangle$, $\left\langle\sigma_{2},  \sigma_{12}\right\rangle$, $\left\langle\sigma_{2},  \sigma_{13}\right\rangle$ and $\left\langle\sigma_{2},  \sigma_{9}\right\rangle$.

Using $\Aa$, we have described eight semigroups - each one isomorphic to a semigroup previously discussed - which contains $\sigma_{5}$  and elements (at least one)  of  $\left\langle\sigma_{6},  \sigma_{9}\right\rangle$.  
 Collections of generators:   $ \left\langle\sigma_{2},  \sigma_{6}, \sigma_{13} \right\rangle$, $\left\langle\sigma_{2},  \sigma_{7}, \sigma_{12} \right\rangle$, $\left\langle\sigma_{2},  \sigma_{8}, \sigma_{11} \right\rangle$, $ \left\langle\sigma_{2},  \sigma_{11}, \sigma_{12} \right\rangle$,  $\left\langle\sigma_{2},  \sigma_{11}, \sigma_{13} \right\rangle$ and $ \left\langle\sigma_{2},  \sigma_{12}, \sigma_{13} \right\rangle$ are minimal in  $\left\langle\sigma_{2},  \sigma_{9} \right\rangle$. 
 Also, collections of generators: 
 $\left\langle\sigma_{5},  \sigma_{7}, \sigma_{12} \right\rangle$,   
  $\left\langle\sigma_{5},  \sigma_{8}, \sigma_{11} \right\rangle$, 
   $ \left\langle\sigma_{5},  \sigma_{9}, \sigma_{10} \right\rangle$,
     $ \left\langle\sigma_{5},  \sigma_{10}, \sigma_{11} \right\rangle$,  
  $\left\langle\sigma_{5},  \sigma_{10}, \sigma_{12} \right\rangle$ 
  and 
  $ \left\langle\sigma_{5},  \sigma_{11}, \sigma_{12} \right\rangle$
  are minimal in  $\left\langle\sigma_{5},  \sigma_{6} \right\rangle$.

\subsection{Extensions of  $\left\langle\sigma_3\right\rangle$ and  $\left\langle\sigma_4\right\rangle$ by elements from the semigroup $\left\langle\sigma_{6},  \sigma_{9}\right\rangle$}

Semigroups $\left\langle\sigma_{3}\right\rangle  = \{\sigma_{3}, \sigma_{7},  \sigma_{11}\}$ and $\left\langle\sigma_{4}\right\rangle  = \{\sigma_{4}, \sigma_{8},  \sigma_{12}\}$ are isomorphic by $\Aa$. They are not monoids. The semigroup $\left\langle\sigma_3\right\rangle$ can be extended using  elements of $\left\langle\sigma_{6},  \sigma_{9}\right\rangle$, in three following ways.

\begin{itemize}
	\item 
$\left\langle\sigma_{3},  \sigma_{6}\right\rangle  =\left\langle\sigma_{3},  \sigma_{10}\right\rangle  = \{\sigma_{3}, \sigma_{6},  \sigma_{7}, \sigma_{10}, \sigma_{11}\}$;
\item 
$\left\langle\sigma_{3},  \sigma_{8}\right\rangle = \left\langle\sigma_{3},  \sigma_{12}\right\rangle = \{\sigma_{3}, \sigma_{6}, \sigma_{7}, \ldots , \sigma_{13}\}$;
\item 
$\left\langle\sigma_{3},  \sigma_{9}\right\rangle = \left\langle\sigma_{3},  \sigma_{13}\right\rangle = \{\sigma_{3}, \sigma_{7},  \sigma_{9}, \sigma_{11}, \sigma_{13}\}$.
	
\end{itemize}
Semigroups $\left\langle\sigma_{3},  \sigma_{6}\right\rangle$ and $ \left\langle\sigma_{3},  \sigma_{9}\right\rangle$ are not isomorphic.  Indeed, suppose $J: \left\langle\sigma_{3},  \sigma_{6}\right\rangle \to  \left\langle\sigma_{3},  \sigma_{9}\right\rangle$ is an isomorphism. Thus, $J$ is the identity on $\left\langle\sigma_{3}\right\rangle$ and $J(\sigma_6)= \sigma_9$ and $J(\sigma_{10})= \sigma_{13}$. This gives a contradiction, since $\sigma_{3} \circ  \sigma_{6}= \sigma_{10}$ and $  \sigma_{3}\circ \sigma_{9} =\sigma_{7}$. 

No semigroup $\left\langle\sigma_{3},  \sigma_{6}\right\rangle$ or $\left\langle\sigma_{3},  \sigma_{9}\right\rangle$  has a minimal collection of generators with three elements, so they give new types of isomorphism. Also,  $\left\langle\sigma_{2},  \sigma_{9}\right\rangle$ is not isomorphic to $\left\langle\sigma_{3},  \sigma_{8}\right\rangle$, since $\left\langle\sigma_{2},  \sigma_{9}\right\rangle$ has a unique pair of generators, but $\left\langle\sigma_{3},  \sigma_{8}\right\rangle$ has two pair of generators.

Using $\Aa$, we get    - isomorphic to previously discussed ones - semigroups
$\left\langle\sigma_{4},  \sigma_{9}\right\rangle  = \left\langle \sigma_{4},  \sigma_{13} \right\rangle  = \{\sigma_{4}, \sigma_{8},  \sigma_{9}\,\sigma_{12}, \sigma_{13}\}$, $\left\langle\sigma_{4},  \sigma_{6}\right\rangle = \left\langle\sigma_{4},  \sigma_{10}\right\rangle = \{\sigma_{4}, \sigma_{6},  \sigma_{8}\,\sigma_{10}, \sigma_{12}\}$ and $\left\langle\sigma_{4},  \sigma_{7}\right\rangle = \left\langle\sigma_{4},  \sigma_{11}\right\rangle = \{\sigma_{4}, \sigma_{6}, \sigma_{7}, \ldots , \sigma_{13}\}$.	
There exist minimal collections of generators, such as follows. $$\left\langle\sigma_{3},  \sigma_{8}\right\rangle = \left\langle\sigma_{3},  \sigma_{6}, \sigma_{9} \right\rangle = \left\langle\sigma_{3},  \sigma_{6}, \sigma_{13} \right\rangle =\left\langle\sigma_{3},  \sigma_{9}, \sigma_{10} \right\rangle = \left\langle\sigma_{3},  \sigma_{10}, \sigma_{13} \right\rangle $$ and $$ \left\langle\sigma_{4},  \sigma_{7}\right\rangle = \left\langle\sigma_{4},  \sigma_{6}, \sigma_{9} \right\rangle =  \left\langle\sigma_{4},  \sigma_{9}, \sigma_{10} \right\rangle = \left\langle\sigma_{4},  \sigma_{6}, \sigma_{13}\right\rangle =  \left\langle\sigma_{4},  \sigma_{10}, \sigma_{13} \right\rangle.$$

\subsection{More generators from the set   $\{\sigma_2, \sigma_{3},  \sigma_4, \sigma_{5}\}$ } Now we check that $\left\langle\sigma_{2},  \sigma_{3}\right\rangle = \{\sigma_{2},  \sigma_{3}, \sigma_{6}, \sigma_{7},  \sigma_{10}\,\sigma_{11}\} $ and $ \left\langle\sigma_{4},  \sigma_{5}\right\rangle = \{ \sigma_{4},  \sigma_{5}, \sigma_{8}, \sigma_{9},  \sigma_{12}\, \sigma_{13}\} =  \Aa [\left\langle\sigma_{2},  \sigma_{3}\right\rangle] $, and also  $\left\langle\sigma_{2},  \sigma_{4}\right\rangle = \{\sigma_{2},  \sigma_{4}, \sigma_{6}, \sigma_{8},  \sigma_{10}\,\sigma_{12}\} $ and $\left\langle\sigma_{3},  \sigma_{5}\right\rangle = \{\sigma_{3},  \sigma_{5}, \sigma_{7}, \sigma_{9},  \sigma_{11}\,\sigma_{13}\}=  \Aa [\left\langle\sigma_{2},  \sigma_{4}\right\rangle] $ are two pairs of isomorphic semigroups which give two new isomorphic type.  Each of these semigroups has six elements, so in $\Ma$ there are five six-elements semigroups of three isomorphic types, since  $\left\langle\sigma_{2}, \sigma_{5}\right\rangle$ has $6$ elements which are idempotents. 

In  $\left\langle\sigma_{3}, \sigma_{4}\right\rangle$ there are seven semigroups which have three generators and  have not two generators. These are   
\begin{enumerate} 
\item 
$\left\langle\sigma_{2},  \sigma_{3}, \sigma_{8}\right\rangle = \left\langle\sigma_{2},  \sigma_{3}, \sigma_{9}\right\rangle =\left\langle\sigma_{2},  \sigma_{3}, \sigma_{12}\right\rangle = \left\langle\sigma_{2},  \sigma_{3}, \sigma_{13}\right\rangle$; 
 \item 
$\left\langle\sigma_{4},  \sigma_{5}, \sigma_{6}\right\rangle = \left\langle\sigma_{4},  \sigma_{5}, \sigma_{7}\right\rangle =\left\langle\sigma_{4},  \sigma_{5}, \sigma_{10}\right\rangle = \left\langle\sigma_{4},  \sigma_{5}, \sigma_{11}\right\rangle $; 
 \item 
$\left\langle\sigma_{2},  \sigma_{4}, \sigma_{7}\right\rangle = \left\langle\sigma_{2},  \sigma_{4}, \sigma_{9}\right\rangle =\left\langle\sigma_{2},  \sigma_{4}, \sigma_{11}\right\rangle = \left\langle\sigma_{2},  \sigma_{4}, \sigma_{13}\right\rangle $; 
	\item 
$\left\langle\sigma_{3},  \sigma_{5}, \sigma_{6}\right\rangle = \left\langle\sigma_{3},  \sigma_{5}, \sigma_{8}\right\rangle =\left\langle\sigma_{3},  \sigma_{5}, \sigma_{10}\right\rangle = \left\langle\sigma_{3},  \sigma_{5}, \sigma_{12}\right\rangle$; 
 	\item 
$\left\langle\sigma_{2},  \sigma_{5}, \sigma_{6}\right\rangle = \left\langle\sigma_{2},  \sigma_{5}, \sigma_{9}\right\rangle =\left\langle\sigma_{2},  \sigma_{5}, \sigma_{11}\right\rangle = \left\langle\sigma_{2},  \sigma_{5}, \sigma_{12}\right\rangle $;
	\item 
$\left\langle\sigma_{2},  \sigma_{3},\sigma_{5}\right\rangle $; 
	\item 
$\left\langle\sigma_{2},  \sigma_{4},\sigma_{5}\right\rangle $.
\end{enumerate}
In later sections,  we will use the symbol $ \cong $ denoting  an isomorphism.

\section{Semigroups which are contained in $\left\langle\sigma_{3},  \sigma_{4}\right\rangle $}

Groups $ \left\langle \sigma_{2} \right\rangle \cong \left\langle \sigma_{5} \right\rangle \cong \left\langle \sigma_{7} \right\rangle \cong \left\langle \sigma_{8} \right\rangle \cong\left\langle \sigma_{10} \right\rangle \cong \left\langle \sigma_{13} \right\rangle$ have one element and are isomorphic.

 Groups $ \left\langle \sigma_{6} \right\rangle \cong \left\langle \sigma_{9} \right\rangle \cong \left\langle \sigma_{11} \right\rangle \cong \left\langle \sigma_{12}\right\rangle$, 
  monoids $ \left\langle \sigma_{2}, \sigma_{10}\right\rangle  \cong \left\langle \sigma_{5}, \sigma_{13}\right\rangle $,
and also semigroups $\left\langle \sigma_{7}, \sigma_{10} \right\rangle \cong \left\langle \sigma_{8}, \sigma_{13} \right\rangle$ and $ \left\langle \sigma_{7}, \sigma_{13} \right\rangle \cong \left\langle \sigma_{8}, \sigma_{10} \right\rangle  $ have two elements and Cayley tables as follows. 
$$  \begin{tabular}{l|ll}
 
  &$A$& $B$     \\   \hline

 $A$ & $B$ & $A$    \\ 
 
 $B$ & $A$ & $B$  \\

\end{tabular}\mbox{; }
\begin{tabular}{l|ll}
 
  &$A$& $B$     \\   \hline

 $A$ & $A$ & $B$    \\ 
 
 $B$ & $B$ & $B$  \\

\end{tabular}\mbox{; } 
\begin{tabular}{l|ll}
 
  &$A$& $B$     \\   \hline

 $A$ & $A$ & $B$    \\ 
 
 $B$ & $A$ & $B$  
\end{tabular} \mbox{ and } 
\begin{tabular}{l|ll}
 
  &$A$& $B$     \\   \hline

 $A$ & $A$ & $A$    \\ 
 
 $B$ & $B$ & $B$ \\
 \end{tabular}. $$

Monoids $ \left\langle \sigma_{2}, \sigma_{6}\right\rangle \cong \left\langle \sigma_{5}, \sigma_{9}\right\rangle$, 
 and also  semigroups $\left\langle \sigma_{2}, \sigma_{7} \right\rangle \cong \left\langle \sigma_{5}, \sigma_{8} \right\rangle$ and $ \left\langle \sigma_{2}, \sigma_{8}\right\rangle  \cong \left\langle \sigma_{5}, \sigma_{7}\right\rangle$ and $\left\langle \sigma_{3}\right\rangle \cong \left\langle \sigma_{4}\right\rangle$ have  three elements and 
Cayley tables  as follows.

$$\begin{tabular}{l|lll}
 
  &$A$& $B$  & $C$   \\   \hline

 $A$ & $A$ & $B$ & $C$   \\ 
 
 $B$ & $B$ & $C$ & $B$ \\ 

 $C$ & $C$ & $B$ & $C$
\end{tabular}\mbox{; }
\begin{tabular}{l|lll}
 
  &$A$& $B$  & $C$   \\   \hline

 $A$ & $A$ & $B$ & $C$   \\ 
 
 $B$ & $C$ & $B$ & $C$ \\ 

 $C$ & $C$ & $B$ & $C$
\end{tabular}\mbox{; }$$ $$
\begin{tabular}{l|lll}
 
  &$A$& $B$  & $C$   \\   \hline

 $A$ & $A$ & $C$ & $C$   \\ 
 
 $B$ & $B$ & $B$ & $B$ \\ 

 $C$ & $C$ & $C$ & $C$
\end{tabular}\mbox{ and }
\begin{tabular}{l|lll}
 
  &$A$& $B$  & $C$   \\   \hline

 $A$ & $B$ & $C$ & $B$   \\ 
 
 $B$ & $C$ & $B$ & $C$ \\ 

 $C$ & $B$ & $C$ & $B$
\end{tabular}. $$  

Semigroups  $  \left\langle \sigma_6, \sigma_{7}\right\rangle \cong \left\langle \sigma_8,  \sigma_{9}\right\rangle$ and $  \left\langle \sigma_6, \sigma_{8}\right\rangle \cong \left\langle \sigma_7,  \sigma_{9}\right\rangle$ and $  \left\langle \sigma_7, \sigma_{8}\right\rangle$ have  four elements and Cayley tables  as follows.
$$ \begin{tabular}{l|llll}
                     &$A$ &$B$   &$C$& $D$      \\ \hline 

$A$  &$C$ &$D$  &$A$ & $B$    \\ 

$B$   &$A$ &$B$  &$C$ & $D$  \\

$C$   &$A$ &$B$&$C$ & $D$   \\ 

$D$   &$C$ &$D$  &$A$ & $B$   \\

\end{tabular} \mbox{; } 
\begin{tabular}{l|llll}
                     &$A$  &$B$   &$C$& $D$      \\ \hline 

$A$  &$C$  &$A$  &$A$  &$C$  \\

$B$   &$D$  &$B$  &$B$  &$D$    \\

$C$   &$A$  &$C$  &$C$  &$A$  \\

$D$  &$B$  &$D$  &$D$  &$B$    \\

\end{tabular}$$ and  $$
\begin{tabular}{l|llll}
               &$A$ &$B$   &$C$& $D$  \\ \hline 

$A$   &$A$ &$C$  &$C$ & $A$   \\ 

$B$    &$D$ &$B$& $B$ & $D$   \\ 

$C$   &$A$ &$C$ &$C$ & $A$   \\ 

$D$  &$D$ &$B$ &$B$ & $D$.   \\ 

\end{tabular}$$

Semigroups  $  \left\langle \sigma_2, \sigma_{11}\right\rangle \cong  \left\langle \sigma_5, \sigma_{12}\right\rangle$ and $  \left\langle \sigma_2, \sigma_{12}\right\rangle \cong \left\langle \sigma_5,  \sigma_{11}\right\rangle$ and $\left\langle \sigma_2,  \sigma_{13}\right\rangle \cong  \left\langle \sigma_5, \sigma_{10}\right\rangle$ and $  \left\langle \sigma_3, \sigma_{6}\right\rangle \cong \left\langle \sigma_4,  \sigma_{9}\right\rangle$ and $  \left\langle \sigma_3, \sigma_{9}\right\rangle \cong \left\langle \sigma_4,  \sigma_{6}\right\rangle$ have 
five elements. 
Semigroups  $  \left\langle  \sigma_2, \sigma_{3}\right\rangle \cong  \left\langle  \sigma_4, \sigma_{5}\right\rangle$ and $  \left\langle \sigma_2, \sigma_4, \right\rangle \cong \left\langle \sigma_3, \sigma_5  \right\rangle$ and $\left\langle \sigma_2, \sigma_5\right\rangle$ have six elements.
Construction of Cayley tables for these semigroups, as well as other semigroups, we leave to the readers. One can prepare such tables removing from the Cayley table for $\Ma$ some columns and rows. Then change Kuratowski operations onto  letters of the alphabet.
For example, $\left\langle \sigma_0, \sigma_2, \sigma_5\right\rangle$ has the following Cayley table $$ \begin{tabular}{l|lllllll}
 &  $0$ &  $A$    &$B$    &$C$      &$D$  &$E$& $F$  \\ \hline 
$0$&$0$& $A$  &$B$  &$C$ &$D$ &$E$ &$F$   \\

$A$& $A$ & $A$ & $C$ &$C$ &$E$ &$E$  &$C$   \\

$B$&  $B$ & $D$ &$B$ &$F$ &$D$ &$D$ &$F$     \\

$C$&  $C$ & $E$ &$C$ &$C$ &$E$ &$E$ &$C$    \\ 

$D$&  $D$ & $D$ &$F$ &$F$ &$D$ &$D$ &$F$   \\

$E$&  $E$ & $E$ &$C$ &$C$ &$E$ &$E$ &$C$    \\

$F$&  $F$ & $D$ &$F$ &$F$ &$D$ &$D$ &$F$.   \\ 

\end{tabular} $$ 
Preparing the Cayley table for  $\left\langle \sigma_0, \sigma_2, \sigma_5\right\rangle$ we put $\sigma_{0}=0, \sigma_{2}=A, \sigma_{5}=B, \sigma_{7}=C, \sigma_{8}=D, \sigma_{10}=E$ and $\sigma_{13}=F$. This table immediately shows that semigroups $\left\langle \sigma_2, \sigma_5\right\rangle$ and $\left\langle \sigma_0, \sigma_2, \sigma_5\right\rangle$ have exactly two automorphisms. These are  identities and   restrictions of $\Aa$.

 Semigroup $  \left\langle  \sigma_6, \sigma_{9}\right\rangle$ has  eight elements. 
Semigroups  $  \left\langle \sigma_2,   \sigma_{9}\right\rangle \cong \left\langle \sigma_5,   \sigma_{6}\right\rangle$ and $  \left\langle \sigma_3,   \sigma_{8}\right\rangle \cong \left\langle \sigma_4,   \sigma_{7}\right\rangle$ have nine elements.
Semigroups  $  \left\langle \sigma_2,   \sigma_3, \sigma_{8}\right\rangle \cong \left\langle \sigma_4,   \sigma_5, \sigma_{7}\right\rangle$ and $  \left\langle \sigma_2, \sigma_4,   \sigma_{7}\right\rangle \cong \left\langle \sigma_3,   \sigma_5, \sigma_{8}\right\rangle$ and $  \left\langle \sigma_2, \sigma_5,   \sigma_{6}\right\rangle $ have ten elements.
 Semigroups  $  \left\langle \sigma_2,   \sigma_3, \sigma_{5}\right\rangle \cong \left\langle \sigma_2,   \sigma_4, \sigma_{5}\right\rangle$ have eleven elements. 
In the end, the semigroup  $  \left\langle \sigma_3, \sigma_4\right\rangle $ has twelve elements.

Thus,   the semigroup  $  \left\langle \sigma_3, \sigma_4\right\rangle $ includes  fifty seven semigroups,  among which are ten groups and fourteen monoids, and also forty three semigroups which are not monoids.    These semigroups consists of twenty eight types of non-isomorphic  semigroups,  two non-isomorphic types of groups, two non-isomorphic types of monoids which are not groups and twenty four non-isomorphic types of semigroups which are not monoids.

\section{Viewing semigroups contained in  $\Ma$}

\subsection{Descriptive data  on semigroups which are contained in $\Ma$} There are one hundred eighteen, i.e. $118= 2 \cdot 57 + 4$, semigroups which are contained in $\Ma$. These are fifty-seven semigroups contained in 
$  \left\langle \sigma_3, \sigma_4\right\rangle $, 
fifty-seven monoids formed by adding $\sigma_0$ 
to a semigroup contained in $  \left\langle \sigma_3, \sigma_4\right\rangle $, 
and  groups 
$  \left\langle \sigma_0 \right\rangle $, $  \left\langle \sigma_1 \right\rangle $, 
and monoids $  \left\langle \sigma_1, \sigma_6\right\rangle= \Ma_1 $ and $  \left\langle \sigma_1, \sigma_2\right\rangle= \Ma $.

There are fifty six types of non-isomorphic  semigroups in $\Ma$. These are twenty eight non-isomorphic types of semigroups contained in $  \left\langle \sigma_3, \sigma_4\right\rangle $, twenty six types of non-isomorphic monoids formed by adding $\sigma_0$ to a semigroup contained in $  \left\langle \sigma_3, \sigma_4\right\rangle $, and also $\Ma_1$ and $\Ma$. Indeed, adding $\sigma_0$ to a semigroup which is not a monoid we obtain a monoid. In this way, we get twenty four types of non-isomorphic monoids. Adding   $\sigma_0$ to a monoid contained in $  \left\langle \sigma_3, \sigma_4\right\rangle $ we get two new non-isomorphic types of monoids.   But, adding   $\sigma_0$ to a group  contained in $  \left\langle \sigma_3, \sigma_4\right\rangle $ we get no new type of monoid, since we get a monoid isomorphic to  $  \left\langle \sigma_2, \sigma_{10}\right\rangle $ or $  \left\langle \sigma_2, \sigma_6\right\rangle $.   The other two types are $\Ma_1$ and $\Ma$.

\subsection{Semigroups which are  not monoids}
 Below we have reproduced, using  the smallest number of generators and the dictionary order, a list of all  43 ,  included in the $ \Ma $. 

\textbf{(1).} $  \left\langle \sigma_2,   \sigma_3 \right\rangle =\{  \sigma_2,   \sigma_3,  \sigma_{6},   \sigma_7,  \sigma_{10},   \sigma_{11} \} $.

\textbf{(2).} $  \left\langle \sigma_2,   \sigma_3, \sigma_{5}\right\rangle = \{ \sigma_2,   \sigma_3, \sigma_5, \sigma_{6},   \ldots,    \sigma_{13}   \}$.

\textbf{(3).}  $  \left\langle \sigma_2,   \sigma_3, \sigma_{8}\right\rangle =  \left\langle \sigma_2,   \sigma_3, \sigma_{9}\right\rangle = \left\langle \sigma_2,   \sigma_3, \sigma_{12}\right\rangle = \left\langle \sigma_2,   \sigma_3, \sigma_{13}\right\rangle =\{ \sigma_2,   \sigma_3, \sigma_6, \sigma_{7},   \ldots,    \sigma_{13}   \} .$ 
 
\textbf{(4).}  $ \left\langle \sigma_{2}, \sigma_{4} \right\rangle = \{ \sigma_2,   \sigma_4, \sigma_6, \sigma_{8},  \sigma_{10},     \sigma_{12}   \}$. 
 
\textbf{(5).}  $  \left\langle \sigma_2,   \sigma_4, \sigma_{5}\right\rangle = \{ \sigma_2,   \sigma_4, \sigma_5,    \ldots,    \sigma_{13}   \} $.

\textbf{(6).}  $  \left\langle \sigma_2,   \sigma_4, \sigma_{7}\right\rangle =  \left\langle \sigma_2,   \sigma_4, \sigma_{9}\right\rangle =  \left\langle \sigma_2,   \sigma_4, \sigma_{11}\right\rangle =  \left\langle \sigma_2,   \sigma_4, \sigma_{13}\right\rangle =  \{ \sigma_2,   \sigma_4, \sigma_6, \sigma_7    \ldots,    \sigma_{13}   \}  $.

\textbf{(7).}  $\left\langle \sigma_{2},   \sigma_{5} \right\rangle =\{ \sigma_{2},   \sigma_{5},  \sigma_{7},   \sigma_{8},  \sigma_{10}, \sigma_{13}\} $. 
 
\textbf{(8).}  $\left\langle \sigma_{2},   \sigma_{5}, \sigma_{6} \right\rangle =\left\langle \sigma_{2},   \sigma_{5}, \sigma_{9} \right\rangle =\left\langle \sigma_{2},   \sigma_{5}, \sigma_{11} \right\rangle =\left\langle \sigma_{2},   \sigma_{5}, \sigma_{12} \right\rangle =\{ \sigma_{2},   \sigma_{5},  \sigma_{6},   \ldots,   \sigma_{13}\}$.

\textbf{(9).}  $\left\langle \sigma_{2},   \sigma_{7} \right\rangle =\{ \sigma_{2},   \sigma_{7},  \sigma_{10}\} $.  

\textbf{(10).}  $\left\langle \sigma_{2},   \sigma_{8} \right\rangle =\{ \sigma_{2},  \sigma_{8},  \sigma_{10}\}$. 

\textbf{(11).}  $\left\langle \sigma_{2},   \sigma_{9} \right\rangle =\left\langle \sigma_{2},   \sigma_{6}, \sigma_{13} \right\rangle =\left\langle \sigma_{2},   \sigma_{7}, \sigma_{12} \right\rangle =\left\langle \sigma_{2},   \sigma_{8}, \sigma_{11} \right\rangle =\left\langle \sigma_{2},   \sigma_{11}, \sigma_{12} \right\rangle =\left\langle \sigma_{2},   \sigma_{11}, \sigma_{13} \right\rangle =\left\langle \sigma_{2},   \sigma_{12}, \sigma_{13} \right\rangle=\{ \sigma_{2},  \sigma_{6},  \sigma_{7},  \ldots , \sigma_{13}\}$. 

\textbf{(12).} $\left\langle \sigma_{2},   \sigma_{11} \right\rangle = \left\langle \sigma_{2},   \sigma_{6}, \sigma_{7} \right\rangle =\{ \sigma_{2},  \sigma_{6},  \sigma_{7},  \sigma_{10}, \sigma_{11}\} $. 

\textbf{(13).} $\left\langle \sigma_{2},   \sigma_{12} \right\rangle =\left\langle \sigma_{2},   \sigma_{6}, \sigma_{8} \right\rangle =\{ \sigma_{2},  \sigma_{6},  \sigma_{8},  \sigma_{10},  \sigma_{12}\} $. 

\textbf{(14).} $\left\langle \sigma_{2},   \sigma_{13} \right\rangle =\left\langle \sigma_{2},   \sigma_{7}, \sigma_{8} \right\rangle =\{ \sigma_{2},  \sigma_{7},  \sigma_{8},   \sigma_{10}, \sigma_{13}\} $.

 \textbf{(15).} $\left\langle \sigma_{3} \right\rangle =\{ \sigma_{3},  \sigma_{7},     \sigma_{11}\}$.

 \textbf{(16).} $ \left\langle \sigma_{3}, \sigma_{4} \right\rangle =\{\sigma_2, \sigma_3, \ldots, \sigma_{13}\}$.   
  
 \textbf{(17).} $ \left\langle \sigma_{3}, \sigma_{5} \right\rangle =\{ \sigma_{3},  \sigma_{5},  \sigma_{7},  \sigma_{9},  \sigma_{11}, \sigma_{13}\}$.
 
 \textbf{(18).} $ \left\langle \sigma_{3}, \sigma_{5}, \sigma_{6} \right\rangle =\left\langle \sigma_{3}, \sigma_{5}, \sigma_{8} \right\rangle =\left\langle \sigma_{3}, \sigma_{5}, \sigma_{10} \right\rangle =\left\langle \sigma_{3}, \sigma_{5}, \sigma_{12} \right\rangle =\{ \sigma_{3},  \sigma_{5},  \sigma_{6},  \ldots ,  \sigma_{13}\}$.
 
 \textbf{(19).} $ \left\langle \sigma_{3}, \sigma_{6} \right\rangle =\left\langle \sigma_{3}, \sigma_{10} \right\rangle =\{ \sigma_{3},  \sigma_{6},  \sigma_{7},  \sigma_{10},  \sigma_{11}\} $.

\textbf{(20).} $ \left\langle \sigma_{3}, \sigma_{8} \right\rangle =\left\langle \sigma_{3}, \sigma_{12} \right\rangle =\{ \sigma_{3},  \sigma_{6},  \sigma_{7},  \ldots ,  \sigma_{13}\}$.

\textbf{(21).} $ \left\langle \sigma_{3}, \sigma_{9} \right\rangle =\left\langle \sigma_{3}, \sigma_{13} \right\rangle =\{ \sigma_{3},  \sigma_{7},  \sigma_{9},  \sigma_{11},  \sigma_{13}\} $.

\textbf{(22).} $\left\langle \sigma_{4} \right\rangle =\{ \sigma_{4},  \sigma_{8},     \sigma_{12}\}$.

\textbf{(23).} $ \left\langle \sigma_{4}, \sigma_{5} \right\rangle =\{ \sigma_{4},  \sigma_{5},  \sigma_{8},  \sigma_{9},  \sigma_{12},  \sigma_{13}\}$.

\textbf{(24).} $ \left\langle \sigma_{4}, \sigma_{5}, \sigma_{6} \right\rangle =\left\langle \sigma_{4}, \sigma_{5}, \sigma_{7} \right\rangle =\left\langle \sigma_{4}, \sigma_{5}, \sigma_{10} \right\rangle =\left\langle \sigma_{4}, \sigma_{5}, \sigma_{11} \right\rangle =\{ \sigma_{4},  \sigma_{5},    \ldots ,  \sigma_{13}\}$.

\textbf{(25).} $ \left\langle \sigma_{4}, \sigma_{6} \right\rangle =\left\langle \sigma_{4}, \sigma_{10} \right\rangle =\{ \sigma_{4},  \sigma_{6},  \sigma_{8},  \sigma_{10},  \sigma_{12}\} $.

\textbf{(26).} $ \left\langle \sigma_{4}, \sigma_{7} \right\rangle = \left\langle \sigma_{4}, \sigma_{11} \right\rangle =\{ \sigma_{4},  \sigma_{6}, \sigma_{7}, \ldots,   \sigma_{13} \}$.

\textbf{(27).} $ \left\langle \sigma_{4}, \sigma_{9} \right\rangle =\left\langle \sigma_{4}, \sigma_{13} \right\rangle =\{ \sigma_{4},  \sigma_{8},  \sigma_{9},  \sigma_{12},  \sigma_{13}\} $. 

\textbf{(28).} $ \left\langle \sigma_{5}, \sigma_{6} \right\rangle = \left\langle \sigma_{5},   \sigma_{7}, \sigma_{12} \right\rangle =\left\langle \sigma_{5},   \sigma_{8}, \sigma_{11} \right\rangle =\left\langle \sigma_{5},   \sigma_{9}, \sigma_{10} \right\rangle =\left\langle \sigma_{5},   \sigma_{10}, \sigma_{11} \right\rangle =\left\langle \sigma_{5},   \sigma_{10}, \sigma_{12} \right\rangle =\left\langle \sigma_{5},   \sigma_{11}, \sigma_{12} \right\rangle=\{ \sigma_{5},  \sigma_{6},    \ldots ,  \sigma_{13}\}$.

\textbf{(29).} $ \left\langle \sigma_{5}, \sigma_{7} \right\rangle =\{ \sigma_{5},  \sigma_{7},  \sigma_{13} \}$.

\textbf{(30).} $ \left\langle \sigma_{5}, \sigma_{8} \right\rangle=\{ \sigma_{5},  \sigma_{8},  \sigma_{13} \} $.

\textbf{(31).} $ \left\langle \sigma_{5}, \sigma_{10} \right\rangle = \left\langle \sigma_{5}, \sigma_{7}, \sigma_8 \right\rangle =\{ \sigma_{5},  \sigma_{7},  \sigma_{8},  \sigma_{10},  \sigma_{13}\} $.

\textbf{(32).} $ \left\langle \sigma_{5}, \sigma_{11} \right\rangle = \left\langle \sigma_{5}, \sigma_{7}, \sigma_9  \right\rangle =\{ \sigma_{5},  \sigma_{7},  \sigma_{9},  \sigma_{11},  \sigma_{13}\} $.

\textbf{(33).} $ \left\langle \sigma_{5}, \sigma_{12} \right\rangle = \left\langle \sigma_{5}, \sigma_{8}, \sigma_9 \right\rangle =\{ \sigma_{5},  \sigma_{8},  \sigma_{9},  \sigma_{12},  \sigma_{13}\} $.

\textbf{(34).} $ \left\langle \sigma_{6}, \sigma_{7} \right\rangle =\left\langle \sigma_{6}, \sigma_{11} \right\rangle =\left\langle \sigma_{10}, \sigma_{11} \right\rangle =\{ \sigma_{6},  \sigma_{7}, \sigma_{10}, \sigma_{11} \}$.

\textbf{(35).} $ \left\langle \sigma_{6}, \sigma_{8} \right\rangle = \left\langle \sigma_{6}, \sigma_{12} \right\rangle =\left\langle \sigma_{10}, \sigma_{12} \right\rangle =\{ \sigma_{6},  \sigma_{8}, \sigma_{10}, \sigma_{12} \} $.

\textbf{(36).} $ \left\langle \sigma_{6}, \sigma_{9} \right\rangle = \left\langle \sigma_{6}, \sigma_{13} \right\rangle = \left\langle \sigma_{7}, \sigma_{12} \right\rangle = \left\langle \sigma_{8}, \sigma_{11} \right\rangle = \left\langle \sigma_{9}, \sigma_{10} \right\rangle = \left\langle \sigma_{11}, \sigma_{12} \right\rangle =
\left\langle \sigma_{6}, \sigma_{7}, \sigma_8 \right\rangle = \left\langle \sigma_{7}, \sigma_{8}, \sigma_9 \right\rangle = \left\langle \sigma_{10}, \sigma_{11}, \sigma_{13} \right\rangle = \left\langle \sigma_{10}, \sigma_{12}, \sigma_{13} \right\rangle =
\{ \sigma_{6},  \sigma_{7}, \ldots,  \sigma_{13} \}   $.

\textbf{(37).} $ \left\langle \sigma_{7}, \sigma_{8} \right\rangle =\{ \sigma_{7},  \sigma_{8}, \sigma_{10},   \sigma_{13} \} $.

\textbf{(38).} $ \left\langle \sigma_{7}, \sigma_{9} \right\rangle =\{ \sigma_{7},  \sigma_{9}, \sigma_{11},   \sigma_{13} \}$.

\textbf{(39).} $ \left\langle \sigma_{7}, \sigma_{10} \right\rangle =\{ \sigma_{7},  \sigma_{10} \} $.

\textbf{(40).} $ \left\langle \sigma_{7}, \sigma_{13} \right\rangle =\{ \sigma_{7},  \sigma_{13} \}$.

\textbf{(41).} $ \left\langle \sigma_{8}, \sigma_{9} \right\rangle =\{ \sigma_{8},  \sigma_{9}, \sigma_{12},  \sigma_{13} \}$.

\textbf{(42).} $ \left\langle \sigma_{8}, \sigma_{10} \right\rangle =\{ \sigma_{8},  \sigma_{10}\}$. 

\textbf{(43).} $ \left\langle \sigma_{8}, \sigma_{13} \right\rangle =\{ \sigma_{8},  \sigma_{13}\}  $.

\subsection{Isomorphic types of semigroups contained in  $\Ma$}
Let systematize the list of all isomorphic types of semigroups contained in the monoid  $\Ma$. Isomorphisms, which are restrictions of the isomorphism $\Aa $, will be regarded as self-evident, and therefore they will not be commented. 

\noindent - The monoid $\Ma$ contains  12 groups with 2 isomorphic types. These are 7 one-element groups and 5 two-element groups; 

\noindent - The monoid $\Ma$ contains  8  two-element monoids with $2$ isomorphic types. These are $\sigma_0$ added to 6 one-element groups and  $ \left\langle \sigma_{2}, \sigma_{10} \right\rangle \cong  \left\langle \sigma_{5}, \sigma_{13} \right\rangle$. Also, it contains 4 two-element semigroups - not monoids, with 2 isomorphic types. These are  $ \left\langle \sigma_{7}, \sigma_{10} \right\rangle \cong  \left\langle \sigma_{8}, \sigma_{13} \right\rangle$ and
$ \left\langle \sigma_{8}, \sigma_{10} \right\rangle \cong  \left\langle \sigma_{7}, \sigma_{13} \right\rangle$;
   
\noindent - The monoid $\Ma$ contains 12  three-element monoids with 5 isomorphic types. These are $\sigma_0$ added to 4 two-element groups and also $ \left\langle \sigma_0, \sigma_{2}, \sigma_{10} \right\rangle \cong  \left\langle \sigma_0, \sigma_{5}, \sigma_{13} \right\rangle$,  $ \left\langle \sigma_0, \sigma_{7}, \sigma_{10} \right\rangle \cong  \left\langle \sigma_0, \sigma_{8}, \sigma_{13} \right\rangle$, $ \left\langle \sigma_0, \sigma_{8}, \sigma_{10} \right\rangle \cong  \left\langle \sigma_0, \sigma_{7}, \sigma_{13} \right\rangle$ and $ \left\langle \sigma_2, \sigma_{6} \right\rangle \cong  \left\langle \sigma_5, \sigma_{9} \right\rangle$;

\noindent - The monoid $\Ma$ contains 6 three-element semigroups - not   monoids,  with 3 isomorphic types. These are   $ \left\langle \sigma_2, \sigma_{7} \right\rangle \cong  \left\langle \sigma_5, \sigma_{8} \right\rangle$, $ \left\langle \sigma_2, \sigma_{8} \right\rangle \cong  \left\langle \sigma_5, \sigma_{7} \right\rangle$ and $ \left\langle \sigma_3 \right\rangle \cong  \left\langle \sigma_4  \right\rangle$;

\noindent - The monoid $\Ma$ contains 8 four-element monoids, each  contains $\sigma_0$,  with 4 isomorphic types. These are semigroups from two preceding items that can be substantially extended by $\sigma_0$;

\noindent - The monoid $\Ma$ contains 5 four-element semigroups - not   monoids,  with 3 isomorphic types. These are   $ \left\langle \sigma_6, \sigma_{7} \right\rangle \cong  \left\langle \sigma_8, \sigma_{9} \right\rangle$, $ \left\langle \sigma_6, \sigma_{8} \right\rangle \cong  \left\langle \sigma_7, \sigma_{9} \right\rangle$ and $ \left\langle \sigma_7, \sigma_8  \right\rangle$. These semigroups extended by $\sigma_0$ yield  5  monoids, all which consist of five elements,  with 3    isomorphic types;

\noindent - The monoid $\Ma$ contains 10 five-element semigroups - not   monoids,  with 5 isomorphic types. These are   $ \left\langle \sigma_2, \sigma_{11} \right\rangle \cong  \left\langle \sigma_5, \sigma_{12} \right\rangle$, $ \left\langle \sigma_2, \sigma_{12} \right\rangle \cong  \left\langle \sigma_5, \sigma_{11} \right\rangle$, $ \left\langle \sigma_2, \sigma_{13} \right\rangle \cong  \left\langle \sigma_5, \sigma_{10} \right\rangle$, $ \left\langle \sigma_3, \sigma_{6} \right\rangle \cong  \left\langle \sigma_4, \sigma_{9} \right\rangle$ and  $ \left\langle \sigma_3, \sigma_{9} \right\rangle \cong  \left\langle \sigma_4, \sigma_{6} \right\rangle$.  These semigroups extended by $\sigma_0$ yield  10  monoids, all which consist of six elements,  with 5    isomorphic types. We get 10 new isomorphic types, since the semigroups are distinguished by semigroups  $ \left\langle \sigma_3 \right\rangle$ and  $ \left\langle \sigma_4\right\rangle$,  and by not isomorphic  semigroups  $ \left\langle \sigma_6, \sigma_{7} \right\rangle$, $ \left\langle \sigma_6, \sigma_{8} \right\rangle$ and $ \left\langle \sigma_7, \sigma_{8} \right\rangle$.

\noindent - The monoid $\Ma$ contains 5 six-element semigroups - not   monoids, with 3 isomorphic types. These are   $ \left\langle \sigma_2, \sigma_{3} \right\rangle \cong  \left\langle \sigma_4, \sigma_{5} \right\rangle$, $ \left\langle \sigma_2, \sigma_{4} \right\rangle \cong  \left\langle \sigma_3, \sigma_{5} \right\rangle$ and  $ \left\langle \sigma_2, \sigma_{5} \right\rangle $.  These semigroups extended by $\sigma_0$ yield   5  monoids, all which consist of seven elements,  with 3    isomorphic types. We get 6 new isomorphic types, since the semigroups  are distinguished by  not isomorphic   semigroups  $ \left\langle \sigma_6, \sigma_{7} \right\rangle$, $ \left\langle \sigma_6, \sigma_{8} \right\rangle$ and $ \left\langle \sigma_7, \sigma_{8} \right\rangle$;

\noindent - The monoid $\Ma$ contains no seven-element semigroup - not a  monoid,   no eight-element monoid 
and the only semigroup   $ \left\langle \sigma_6, \sigma_{9} \right\rangle$ with exactly eight elements and the only monoid    $ \left\langle \sigma_0, \sigma_6, \sigma_{9} \right\rangle$ with exactly nine elements;
 
\noindent - The monoid $\Ma$ contains 4 nine-element semigroups - not   monoids,  with 2 isomorphic types. These are   $ \left\langle \sigma_2, \sigma_{9} \right\rangle \cong  \left\langle \sigma_5, \sigma_{6} \right\rangle$ and  $ \left\langle \sigma_3, \sigma_8 \right\rangle \cong  \left\langle \sigma_4, \sigma_7 \right\rangle$. The semigroup  $ \left\langle \sigma_2, \sigma_{9} \right\rangle$ does not contain a semigroup isomorphic to  $\left\langle \sigma_3 \right\rangle$, hence it is not isomorphic to $ \left\langle \sigma_3, \sigma_{8} \right\rangle$.  These semigroups extended by $\sigma_0$ yield  4  monoids, all which  consist of ten elements,  with 2    isomorphic types;

\noindent - The monoid $\Ma$ contains 6 ten-element semigroups - not   monoids,  with 4 isomorphic types. These are $\left\langle \sigma_1, \sigma_{6} \right\rangle$,  $ \left\langle \sigma_2, \sigma_{3}, \sigma_8 \right\rangle \cong  \left\langle \sigma_4, \sigma_{5}, \sigma_6 \right\rangle$, $ \left\langle \sigma_2, \sigma_{4}, \sigma_7 \right\rangle \cong  \left\langle \sigma_3, \sigma_{5}, \sigma_6 \right\rangle$ and  $ \left\langle \sigma_2, \sigma_5, \sigma_6 \right\rangle $. These semigroups (except $\left\langle \sigma_1, \sigma_{6} \right\rangle$)   extended by $\sigma_0$ yield  5  monoids, all which  consist of ten elements,  with 3    isomorphic types. We get 6 new isomorphic types, since the semigroups  are distinguished by  not isomorphic   semigroups  $ \left\langle \sigma_2, \sigma_{3} \right\rangle$, $ \left\langle \sigma_2, \sigma_{4} \right\rangle$ and $ \left\langle \sigma_2, \sigma_{5} \right\rangle$;

\noindent - The monoid $\Ma$ contains 2 isomorphic semigroups - not   monoids,  which consist of 11 elements i.e.,    $ \left\langle \sigma_2, \sigma_{3}, \sigma_5 \right\rangle \cong  \left\langle \sigma_2, \sigma_{4}, \sigma_5 \right\rangle$. These semigroups extended by $\sigma_0$ yield  2 isomorphic  monoids, which consist of 12 elements. The monoid $\Ma$ contains no larger semigroup with the exception of itself.

\section{Cancellation rules motivated by some  topological properties.} 

\subsection{Some consequences of the axiom $\emptyset = \emptyset^-$} So far, we used only following relations (above named cancellation rules): 
$\sigma_2 \circ \sigma_{2} = \sigma_2$, $\sigma_1 \circ \sigma_{1} = \sigma_0$, $\sigma_2 \circ \sigma_{12} = \sigma_6$  and $ \sigma_2 \circ \sigma_{13} = \sigma_7$. When one assumes $X\not=\emptyset = \emptyset^-$, then 
$$X=\sigma_0(X)=\sigma_2(X)= \sigma_5(X)=\sigma_7(X)=\sigma_8(X)=\sigma_{10}(X)=\sigma_{13}(X)$$ and $$\emptyset=\sigma_1(X)=\sigma_3(X)=\sigma_4(X)=\sigma_6(X)=\sigma_9(X)=\sigma_{11}(X)=\sigma_{12}(X).$$
 Using the substitution $A \mapsto A^c$, one obtains  equivalent relations between operation from  the set $\{\sigma_1, \sigma_3, \sigma_4,\sigma_6,\sigma_9, \sigma_{11},\sigma_{12}\}$, and conversely. Therefore, cancellation rules are 
topologically reasonable  only between the operations from the monoid $$ \left\langle \sigma_0, \sigma_{2}, \sigma_5 \right\rangle = \{\sigma_0,\sigma_2, \sigma_5, \sigma_7, \sigma_8, \sigma_{10}, \sigma_{13}\} .$$
 T. A. Chapman, see \cite{chap}, consider properties of subsets with respect to such relations. Below, we are going  to identify relations that are determined by some topological spaces, compare \cite{au} and \cite{auc}. 

\subsection{The relation $\sigma_0=\sigma_2$} If a topological space $X$ is discrete, then there exist two Kuratowski operation, only. These are $\sigma_0$ and  $\sigma_1$. So, the monoid of Kuratowski operations reduced to the group $\left\langle \sigma_1 \right\rangle $.

The relation  $\sigma_0=\sigma_2$ is equivalent to any  relation  $\sigma_0=\sigma_i$, where $i\in \{5, 7, 8 , 10.13\}$. Any such relation implies that every subset of $X$ has to be closed and open. However, one can check these using (only) facts that  $\sigma_1$ is an involution and $
\sigma_2$ is an idempotent and the cancellation rules, i.e. the Cayley table for $\Ma$. So,  $\sigma_0=\sigma_2$ follows $$\sigma_0=\sigma_2 =\sigma_5=\sigma_7 =\sigma_8=\sigma_{10} =\sigma_{13}.$$

\subsection{The relation  $\sigma_2=\sigma_5$} Topologically,  $\sigma_2=\sigma_5$ means that $X$ must be discrete.  
This is so because $A^{c-c} \subseteq A \subseteq A^-$ for any $A\subseteq X$.

\subsection{The relation  $\sigma_2=\sigma_7$} Topologically,  the relation $\sigma_2=\sigma_7$ implies $\sigma_0=\sigma_2$. 
But it requires the use of topology axioms $\emptyset =\emptyset^-$ and $C^- \cup B^-= (C\cup B)^-$ for each $C$ and $ B$.  
\begin{lem} \label{33} For any topological space $\sigma_2=\sigma_7$ implies $ \sigma_2=\sigma_8$.
\end{lem} \begin{proof} If $\sigma_2=\sigma_7$, then $A\not=\emptyset \Rightarrow A^{c-c}\not= \emptyset,$ for any $A\subseteq X$. Indeed, if $A^{c-c}= \emptyset$, then  $\sigma_7(A)=\emptyset^-=\emptyset$. Since $A\not=\emptyset$, then $\sigma_2(A)\not=\emptyset$. Hence $\sigma_2(A)\not=\sigma_7(A)$, a contradiction. 

The axiom $C^- \cup B^-= (C\cup B)^-$  implies that always $$(A^- \cap A^{-c-})^{c-c}= \emptyset .$$ Thus, the additional assumption  $\sigma_2=\sigma_7$ follows that always $A^- \cap A^{-c-}= \emptyset .$ Therefore, always $A^{-c-}=A^{-c}$, but this means that any closed set has to be open: in other words,  $\sigma_2=\sigma_8$. 
\hfill\end{proof}
  
\begin{pro}  For any topological space $\sigma_2=\sigma_7$ implies $ \sigma_0=\sigma_2$
\end{pro} \begin{proof} But the relation  $\sigma_2=\sigma_7$ is equivalent with $\sigma_5=\sigma_8$. By lemma \ref{33}, we get $\sigma_2=\sigma_5$. Finally $\sigma_0=\sigma_2$. 
\hfill\end{proof}

The relation    $\sigma_2=\sigma_7$ has interpretation without the axiom $\emptyset =\emptyset^-$. Indeed, suppose
$X=\{a, b\}$. Put  $$\sigma_2(\emptyset) =\{a\}=\sigma_2(\{a\}) \mbox{  and }  X=\sigma_2(X)=\sigma_2(\{b\}).$$ Then,  check  that    $\sigma_2=\sigma_7$  and   $$ \sigma_8(\emptyset) =\sigma_5(\{a\}) =\sigma_4(\{b\}) =\sigma_1(X) =\emptyset:$$ in other words,  $\sigma_2=\sigma_7$  and $\sigma_2\not=\sigma_8$. However, $\sigma_2=\sigma_7$ is equivalent to $\sigma_5=\sigma_8$. 

This relation implies $\sigma_2=\sigma_7= \sigma_{10}$ and $\sigma_5=\sigma_8= \sigma_{13}$ and $\sigma_3=\sigma_6= \sigma_{11}$ and $\sigma_4=\sigma_9= \sigma_{12}$. For this interpretation, the monoid $\Ma/{(\sigma_2=\sigma_7)}$ - consisting of Kuratowski operation over a such $X$, has six elements, only.
In $\Ma/{(\sigma_2=\sigma_7)}$, there are relations  covered by the following proposition, only.
\begin{pro} For any monoid with the Cayley table  as for $\Ma$, the relation $\sigma_2=\sigma_7$ implies:  
\begin{itemize}
	\item $ \sigma_{2}=\sigma_{7} = \sigma_{10}=\sigma_{7}\circ \sigma_2$;
		\item $ \sigma_5=\sigma_{1}\circ \sigma_2\circ \sigma_1=\sigma_{1}\circ \sigma_7\circ \sigma_1=\sigma_8 = \sigma_5\circ \sigma_2=\sigma_5\circ \sigma_7 =\sigma_{13}$;
			\item $ \sigma_3=\sigma_{2}\circ \sigma_1 = \sigma_{7}\circ \sigma_1 =\sigma_6 = \sigma_{10}\circ \sigma_1 = \sigma_{11}$;
				\item $ \sigma_4=\sigma_{1}\circ \sigma_2 = \sigma_{1}\circ \sigma_7 =\sigma_9 =\sigma_{1}\circ \sigma_{10} = \sigma_{12}$. \hfill $\Box$ \end{itemize}
\end{pro}

Thus, the Cayley table does not contain the complete information resulting from the axioms of topology.

\subsection{The relation  $\sigma_2=\sigma_8$} Topologically,  the relation $\sigma_2=\sigma_8$ means that any closed set is  open, too. Thus, if $X=\{a,b,c\}$ is a topological space with the open sets $X$, $\emptyset$,  $\{a,b\}$ and $\{c\}$, then  $\Ma/{(\sigma_2=\sigma_8)}$ is the monoid of all Kuratowski operations over $X$. Relations $\sigma_2=\sigma_8$ and $\sigma_5=\sigma_7$ are equivalent. They imply  relations: $\sigma_2=\sigma_{8}=\sigma_{10}$, $\sigma_5=\sigma_{7}=\sigma_{13}$,  $ \sigma_3=\sigma_{9}=\sigma_{11}$ and  $\sigma_4=\sigma_{6}=\sigma_{12}$.  
The permutation  
$$ \left(
\begin{tabular}{lllllll}
                   $\sigma_{0}$ & $\sigma_{1}$ &$\sigma_{2}$&$\sigma_{3}$  &$\sigma_{4}$  &$\sigma_{5}$       \\ 
                   $\sigma_{0}$ & $\sigma_{1}$  & $\sigma_{5}$ & $\sigma_{4}$ &$\sigma_{3}$  & $\sigma_{2}$    
\end{tabular}\right)
$$ determines the isomorphism between monoids  $\Ma/{(\sigma_2=\sigma_7)}$ and $\Ma/{(\sigma_2=\sigma_8)}$.

\subsection{The relations  $\sigma_2=\sigma_{10}$ and $\sigma_2=\sigma_{13}$ } Topologically,  the relation $\sigma_2=\sigma_8$ means that any non-empty closed set has non-empty interior. For each $A$, the closed set $A^- \cap A^{-c-}$ has empty interior, so the relation $\sigma_2=\sigma_{10}$ implies  that $A^{-c}$ is closed. Hence, any open set has to be closed, what implies  $\sigma_2=\sigma_{8}$.  The relations   $\sigma_2=\sigma_{13}$ follows that each closed set has to be open, so it  implies  $\sigma_2=\sigma_{8}$, too.

\subsection{The relation  $\sigma_7=\sigma_{8}$} Using the Cayley table for $\Ma$, one can check that  the relation  $\sigma_7=\sigma_{8}$ and $\sigma_{10}=\sigma_{13}$ are equivalent. Each of them gives
	 $\sigma_7=\sigma_{8} =\sigma_{10}=\sigma_{13}$ and $\sigma_6=\sigma_{9} =\sigma_{11}=\sigma_{12}$.
 If $X=\{a,b\}$ is a topological space with the open sets $X$, $\emptyset$ and $\{a\}$, then  $\Ma/{(\sigma_7=\sigma_8)}$ is the monoid of all Kuratowski operations over $X$ and  consists of $8$ elements. 
 
 \subsection{The relation  $\sigma_7=\sigma_{10}$} Using the Cayley table for $\Ma$, one can check that  the relations  $\sigma_7=\sigma_{10}$,  $\sigma_8=\sigma_{13}$,  
	  $\sigma_{6}=\sigma_{11}$ and $\sigma_{9} =\sigma_{12}$ are equivalent.   
	If $X$  is a sequence converging to the point $g$ and $g\in X$, then 
 $\Ma/{(\sigma_7=\sigma_{10})}$ is the monoid of all Kuratowski operations over $X$  and  consists of $10$ elements. 

\subsection{The relation  $\sigma_7=\sigma_{13}$} Using the Cayley table for $\Ma$, one can check that  the relations  $\sigma_7=\sigma_{13}$ and  $\sigma_8=\sigma_{10}$  are equivalent.  Also  $  \sigma_{6} =  \sigma_{12}$ and $    \sigma_{9} =  \sigma_{11}$. These relations give the  monoid  with the following Cayley table, where the row and column marked by the identity  are omitted.

$$ \begin{tabular}{l|lllllllll}   
 &$\sigma_{1}$& $\sigma_{2}$  &$\sigma_{3}$  &$\sigma_{4}$ &$\sigma_{5}$ &$\sigma_{6}$ &$\sigma_{7}$ &$\sigma_{8}$ &$\sigma_{9}$    \\ \hline 
 $\sigma_{1}$ & $\sigma_{0}$ & $\sigma_{4}$ &$\sigma_{5}$ &$\sigma_{2}$ &$\sigma_{3}$ &$\sigma_{8}$ &$\sigma_{9}$ &$\sigma_{6}$ &$\sigma_{7}$   \\ 

 $\sigma_{2}$ & $\sigma_{3}$ & $\sigma_{2}$ &$\sigma_{3}$ &$\sigma_{6}$ &$\sigma_{7}$ &$\sigma_{6}$ &$\sigma_{7}$ &$\sigma_{8}$ &$\sigma_{9}$   \\  

 $\sigma_{3}$ & $\sigma_{2}$ & $\sigma_{6}$ &$\sigma_{7}$ &$\sigma_{2}$ &$\sigma_{3}$ &$\sigma_{8}$ &$\sigma_{9}$ &$\sigma_{6}$ &$\sigma_{7}$    \\  

 $\sigma_{4}$ & $\sigma_{5}$ & $\sigma_{4}$ &$\sigma_{5}$ &$\sigma_{8}$ &$\sigma_{9}$ &$\sigma_{8}$ &$\sigma_{9}$ &$\sigma_{6}$ &$\sigma_{7}$   \\  

 $\sigma_{5}$ & $\sigma_{4}$ & $\sigma_{8}$ &$\sigma_{9}$ &$\sigma_{4}$ &$\sigma_{5}$ &$\sigma_{6}$ &$\sigma_{7}$ &$\sigma_{8}$ &$\sigma_{9}$    \\  

 $\sigma_{6}$ & $\sigma_{7}$ & $\sigma_{6}$ &$\sigma_{7}$ &$\sigma_{8}$ &$\sigma_{9}$ &$\sigma_{8}$ &$\sigma_{9}$ &$\sigma_{6}$ &$\sigma_{7}$   \\  

 $\sigma_{7}$ & $\sigma_{6}$ & $\sigma_{8}$ &$\sigma_{9}$ &$\sigma_{6}$ &$\sigma_{7}$ &$\sigma_{6}$ &$\sigma_{7}$ &$\sigma_{8}$ &$\sigma_{9}$   \\  

 $\sigma_{8}$ & $\sigma_{9}$ & $\sigma_{8}$ &$\sigma_{9}$ &$\sigma_{6}$ &$\sigma_{7}$ &$\sigma_{6}$ &$\sigma_{7}$ &$\sigma_{8}$ &$\sigma_{9}$    \\  
 $\sigma_{9}$ & $\sigma_{8}$ & $\sigma_{6}$ &$\sigma_{7}$ &$\sigma_{8}$ &$\sigma_{9}$ &$\sigma_{8}$ &$\sigma_{9}$ &$\sigma_{6}$ &$\sigma_{7}$      

\end{tabular}  $$
 
If  a space $X$ is  extremally disconnected, then   the closures of open sets are open, compare \cite[p. 452]{eng}.
It follows that $\sigma_6=\sigma_{12}$. The space  $X=\{a,b\}$  with the open sets $X$, $\emptyset$ and $\{a\}$ is extremally disconnected. But it contains an one-element open and dense set $\{ a \}$ and it follows that  $\sigma_7=\sigma_8$. Similar is for the space $\beta N$,  see \cite[p. 228 and 453]{eng}  to find the definition and properties of $\beta N$. There are Hausdorff  extremally disconnected   and dense in itself spaces. 
 For example, the Stone space of the complete Boolean algebra of all regular
closed subsets of the unit interval, compare \cite{gl}. For such spaces $\sigma_7 \not=\sigma_{8}$ and $\sigma_7=\sigma_{13}$. 
To see this, suppose a Hausdorff $X$ is  extremally disconnected and dense in itself. 
Let $X = U\cup V \cup W,$ where sets $U,V$ and $W$ are closed-and-open. Consider a set  $A= A^{c-c} \cup B \cup C$ such that:  
\begin{itemize}
	\item $C^{c-c} =\emptyset $ and  $ C^-=W$;
	\item $\emptyset \not= B \subseteq V$ and $B^{-c-c}= \emptyset$;
	 \item $U=A^{c-c-} \not=A^{c-c}$.
\end{itemize}
  Then check that:  
\begin{itemize}
	\item  $\sigma_0(A)=A$ and  
$\sigma_1(A)= X \setminus (A^{c-c} \cup  B \cup  C)$;
  \item $\sigma_2(A)=U \cup  B^- \cup W$ and 
   $\sigma_3(A)=X \setminus A^{c-c}$; 
   \item   $\sigma_4(A)=V \setminus   B^- $ and 
   $\sigma_5 (A) = A^{c-c}$;
   \item    $\sigma_6(A)=\sigma_{12}(A)=V$ and 
$\sigma_7(A)=\sigma_{13}(A)=U$; \item 
    $\sigma_8(A)=\sigma_{10}(A)=U \cup W$ and 
 $\sigma_9(A)=\sigma_{11}(A)=V \cup W$.
  \end{itemize} 
  Hence we have that $\sigma_7\not=\sigma_{8}$. Note that, if $W=\emptyset$, then $\sigma_7(A) =\sigma_{8}(A)$.  This is the case of subsets of $\beta N$.

\end{document}